\documentclass[11pt,a4paper]{amsart}
\usepackage{graphicx}
\usepackage{amsmath,amsfonts}
\usepackage{amsthm,amssymb,latexsym}
\usepackage[active]{srcltx}
\usepackage[usenames]{color}
\usepackage[utf8]{inputenc} 
\DeclareMathOperator{\rank}{rank}

\newcommand{\bfs}{\boldsymbol}

\newtheorem{theorem}{Theorem}[section]
\newtheorem{corollary}[theorem]{Corollary}
\newtheorem{lemma}[theorem]{Lemma}

\newtheorem*{fact*}{Fact}
\newtheorem{proposition}[theorem]{Proposition}
\theoremstyle{definition}

\newtheorem{remark}[theorem]{Remark}
\numberwithin{equation}{section}

\newcommand{\N}{\mathbb N}

\newcommand{\A}{\mathbb A}
\newcommand{\F}{\mathbb F}
\newcommand{\K}{\mathbb K}

\newcommand{\Pp}{\mathbb P}

\newcommand{\fp}{\F_{\hskip-0.7mm p}}
\newcommand{\fq}{\F_{\hskip-0.7mm q}}

\newcommand{\cfq}{\overline{\F}_{\hskip-0.7mm q}}
\def\ifm#1#2{\relax \ifmmode#1\else#2\fi}

\newcommand{\klk}    {\ifm {,\ldots,} {$,\ldots,$}}

\newcommand{\wt}    {\ifm {{\sf wt}} {{$\sf wt$}}}

\begin{document}

\title[Variants of diagonal equations over finite fields]{Estimates on the number of $\fq$--rational solutions of variants of diagonal equations over finite fields}
%
%
%
\author[M. P\'erez]{
Mariana P\'erez${}^{1,3}$}
\author[M. Privitelli]{
Melina Privitelli${}^{2,3}$}

\address{${}^{1}$ Consejo Nacional de Investigaciones Científicas y Técnicas (CONICET),
Ar\-gentina}
\address{${}^{2}$
Universidad Nacional de Gene\-ral Sarmiento, Instituto de Ciencias, J.M. Guti\'errez 1150
(B1613GSX) Los Polvorines, Buenos Aires, Argentina}
\email{mprivite@ungs.edu.ar}
\address{${}^{3}$
Universidad Nacional de Hurlingham, Instituto de Tecnología e Ingeniería, Av. Gdor. Vergara 2222
(B1688GEZ) Villa Tesei, Buenos Aires, Argentina}
\email{mariana.perez@unahur.edu.ar}

\thanks{The authors were partially supported by the grants
PIP CONICET 11220130100598, PIO CONICET-UNGS 14420140100027 and ICI-UNGS 30/1146}

\keywords{Finite fields, symmetric
polynomial, singular locus, rational solutions, diagonal equations}%
\subjclass[2010]{11T06, 05E05, 14G05, 14G15, 11G25}

\begin{abstract}
In this paper we study the set of  $\fq$-rational solutions of equations defined by polynomials evaluated in  power-sum  polynomials with coefficients in $\fq$. This is done by means of applying a methodology which relies on the study of the geometry of the set of common zeros of symmetric polynomials over the
algebraic closure of $\fq$. We provide improved estimates and existence results of $\fq$-rational solutions  to  the following equations: {\it{deformed diagonal equations}}, {\it{generalized Markoff-Hurwitz-type equations}} and {\it{Carlitz's equations}}. We extend these techniques to  more general variants of diagonal equations over finite fields.
\end{abstract}

\maketitle

\section{Introduction}

Several problems of coding theory, cryptography and combinatorics require the study of
the set of rational points of varieties defined over a finite field $\fq$ on which the symmetric
group of permutations of the coordinates acts. In coding theory, deep holes in the standard
Reed–Solomon code over $\fq$ can be expressed in terms of the set of zeros with coefficients in $\fq$ of
certain symmetric polynomials associated to the code  (see, e.g., \cite{ChMu07} or  \cite{CaMaPr12}). In cryptography, the characterization of monomials defining an almost perfect nonlinear polynomial or
a differentially uniform mapping can be reduced to estimate the number of $\fq$–rational zeros of some symmetric polynomials (see, e.g., \cite{Ro09} or \cite{AuRo10}). Finally,
 several applications in combinatorics over finite fields,
 such as the determination
of the average cardinality of the value set and the distribution of factorization patterns of
families of univariate polynomials with coefficients in $\fq$, has also been expressed in terms
of the number of common $\fq$–rational zeros of symmetric polynomials defined over $\fq$ (see
\cite{CeMaPePr14} and \cite{CeMaPe17}).
In \cite{CaMaPr12}, \cite{CeMaPePr14}, \cite{MaPePr14}, \cite{CeMaPe17} and \cite{MaPePr19} we have  developed a methodology to deal with
some of the problems mentioned above. This methodology relies on the study of the geometry of the set of common zeros of the symmetric polynomials under consideration over the
algebraic closure of $\fq$. By means of such  study we were able to prove that, in all
the cases, the set of common zeros in $\fq$ of the involved
polynomials is a complete intersection whose singular locus has a “controlled” dimension.
This  allowed us to apply certain explicit estimates on the number of $\fq$–rational zeros
of projective complete intersections defined over $\fq$  to obtain a conclusion for the problem under consideration (see, e.g., \cite{GhLa02a}, \cite{CaMa07}, \cite{CaMaPr15}
or \cite{MaPePr16}).

The purpose of this article is twofold. On one hand, we apply our techniques to the problem of estimating the number of $\fq$--rational solutions of certain variants of diagonal equations. On the other hand, we present a more general framework of the theory by extending the aforementioned methodology to a wider class of varieties defined by polynomials evaluated in symmetric polynomials.

We shall consider three well known classes of polynomial equations over finite fields: {\it{deformed diagonal equations}}, {\it{generalized Markoff-Hurwitz-type equations}} and {\it{Carlitz's equations}}.
All previous results on the number of $\fq$--rational solutions of these types of equations rely on techniques of combinatorial analysis, so our approach to the study these problems is novel.

\textbf{Deformed diagonal equations.} Given $g \in \fq[X_1,\ldots,X_n]$, a deformed diagonal equation is an equation of the type 
\begin{equation*} f:= c_1X_1^{m}+ \cdots+c_n X_n^{m}+ g=0, \end{equation*}
where $c_i \in \fq \setminus \{0\}$ and  $\deg(g)<m$. In comparison with diagonal equations, which correspond to the case where $g$ is a constant polynomial, there are much fewer results about the number of $\fq$-solutions of deformed diagonal equations. In \cite{Ca56}, L. Carlitz provides a result which garantees the existence of an $\fq$-rational solution of a deformed diagonal equation when   $\mathrm{char}(\fq)$ divides   $m-1$.  Later, in 2006,  B. Felszeghy \cite{Fe06}  extends Carlitz's result by dropping  the requirement over $\mathrm{char}(\fq)$, but which holds only for prime fields.  Another generalization was provided by Castro et. al. \cite{CaRuVe08} for the case when the exponents are not necessarily equal by computing the exact $p$-divisibility of certain exponential sums. On the other hand, A.  Adolphson and  S. Sperber \cite{AdSp88} used Newton polyhedra
 to prove a result that derives an estimate on the number of $\fq$--rational solutions of this type  of equations. In this article we improve the existence results in the literature imposing no restriction to the characteristic of the field and extending Felszeghy's for several cases. Furthermore, we give an estimate on the number $N_g$ of $\fq$--rational solutions of these type of equations: we prove that  $N_g=q^{n-1}+\mathcal{O}(q^{n/2})$, providing  an explicit expression for the constant underlying the $\mathcal{O}$ notation. 
 
 \textbf{Generalized Markoff-Hurwitz-type equations.} These are of the form
\begin{equation*}\label{Baoulina}
(a_1X_1^{m_1}+\cdots +a_nX_n^{m_n}+a)^k=bX_1^{k_1}\ldots X_n^{k_n},
\end{equation*}
 where $m, n, k_1, \ldots, k_n$, $k$ are positive integers, $a, b \in \fq$ and $a_i \in \fq\setminus \{0\}$, $1\leq i \leq n$. Generalized Markoff-Hurwitz-type equations have been well studied in the special case  $a=0$. Moreover, several papers provide explicit formulas of the number of $\fq$--rational solutions of this type of equations in very particular cases (see, e.g., \cite{Ba15}). In this article, we concentrate in the case $a\neq 0$ and $k=1$. More precisely, we obtain an estimate on $N^*$, the number of $\fq$--rational solutions with the condition that $x_1\cdots x_n \neq 0$. This estimate improves Mordell 's  \cite{Mo63} for the case $m_1=\cdots = m_n$ since our result holds without conditions on the characteristic of the field and we require that $m>k_1+\cdots+k_n$ instead of  $k_1= \cdots= k_n=1$. 
Also, our estimate improves Mordell 's  by determining an extra term in the asymptotic development of $N^*$ in terms of $q$.


\textbf{Carlitz 's equations.} These are of the form
\begin{equation*}
h_1(X_1)+\cdots +h_n(X_n)=g,
\end{equation*} 
where $h_i=a_{d,i} T^d+\cdots +a_{0,i} \in \fq[T]$, $\deg(h_i)=d$ for $1\leq i \leq n$ and $g\in \fq[X_1,\ldots,X_n]$ is such that $\deg (g)< d$. In this case, we provide an explicit estimate on the number $N$ of $\fq$--solutions of this type of equations which implies that $N=q^{n-1}+\mathcal{O}(q^{n/2})$. We also provide an explicit upper bound for the constant underlying the $\mathcal{O}$ notation in terms of $d$ and $n$. This result improves Carlitz's estimate $N=q^{n-1}+\mathcal{O}(q^{n-w})$,
where $w=\frac{1}{kd}$ and the constant implied by the $\mathcal{O}$ is not explicitly given (see \cite{Ca53}). Moreover, Carlitz's result only holds when $g$ is a constant polynomial. We also obtain an existence result which improves Carlitz's in several aspects. Finally, as a particular case of this type of equations, we study the Dickson's equations.

All the aforementioned results complement those existing  in the literature.

 Our other main objective is the extension of our geometric methodology to a more general type of equations: those given by polynomials evaluated in power-sum polynomials. More precisely, let $P_{m_j}=X_1^{m_{j}}+\cdots+X_n^{m_{j}}$ be the $m_j$-power sum polynomial of $\fq[X_1,\ldots,X_n]$ and  $g\in \fq$ or $g\in \fq[X_1,\ldots,X_n]$ be a polynomial with $deg(g)< c(m_1 \klk m_d)$, where $c(m_1 \klk m_d)$ is a constant which depends on $m_1 \klk m_d$ .
%
%
  We consider  new indeterminates  $Y_1,\ldots,Y_d$ over $\fq$  and $f\in \fq[Y_1,\ldots,Y_d]$. Our purpose is to estimate the number of $\fq$--solutions of the following type of equations:
  $$f(P_{m_1} \klk P_{m_d})+ g=0.$$ 
   We quickly outline our methodology to provide insight on the ideas behind it.
  Let  $\wt :\fq[Y_1,\ldots,Y_d] \longrightarrow \N$ be the weight defined by setting $\wt(Y_j) :=m_{j}$ for $1\leq j \leq d$ and let $\wt(f)$ be the \textit{weight} of $f$.
  The equation above can be rewritten in the following way: 
 $$f(P_{m_1} \klk P_{m_d})+ X_1^{\wt(f)}+\cdots + X_n^{\wt(f)}+g_1=0,$$
  where $g_1\in \fq[X_1,\ldots,X_n]$ and $\deg(g_1)<\wt(f).$
  We concentrate on the more general problem of estimating the number of $\fq$--rational solutions of the equation
  $$f(P_{m_1} \klk P_{m_d})+ X_1^{e}+\cdots + X_n^{e}+g_1=0,$$
  where $g_1\in \fq[X_1,\ldots,X_n]$ and $\deg(g_1)<e$ for \textit{any}  positive integer $e$.  
  Let $f^{\wt}$ stand for the component
of highest weight of $f$ and 
let $\nabla f$ and $\nabla f^{\wt}$ be
 the gradients of $f$ and $f^{\wt}$ respectively.
 Suppose that the following hypotheses hold:
 \begin{enumerate}
 \item[$(H_1)$] $\nabla f\neq 0$ on every point of $ \A^d$,
\item[$(H_2)$] $\nabla f^{\wt}\neq 0$ on every point of $ \A^d$,
 \end{enumerate}
 and let $R_g$ be the polynomial
$R_g:=f(P_{m_1} \klk P_{m_d})+  X_1^{e}+\cdots + X_n^{e}+g_1$.
We shall study the geometric properties of the  affine hypersurfaces $V_g:=V(R_g)\subset\A^n$ with similar arguments as those in the papers cited above. In order to estimate the number of $\fq$-rational points of $V_g$ we consider $\mathrm{pcl}(V_g)$, the projective closure of $V_g$. We shall provide a suitable bound of the dimension of the singular locus of $\mathrm{pcl}(V_g)$ which allows us to prove that $\mathrm{pcl}(V_g)$ is absolutely irreducible. Then, applying estimates for absolutely irreducible singular projective varieties  \cite{GhLa02a}, can we provide the main result of this part. 
\begin{theorem}\label{teo 1 intro}
 Let $d,n, e$ be positive integers such that $1 \leq d \leq n-3$. Let $m_{1},\ldots, m_{d}$ be positive integers with  $ m_{1}<\cdots<  m_{d}$. Assume that $\mathrm{char}(\fq)$ does not divide $e$ and  $m_{j}$ for all $1\leq j \leq d$. Let $V_g=V(R_g) \subset \A^n$ be the variety defined by $R_g=f(P_{m_1} \klk P_{m_d})+g$ with $f \in  \fq[Y_1 \klk Y_d]$ and  $g \in \fq[X_1 \klk X_n]$ of degree $e$ defined as $g= X_1^e+\cdots + X_n^e+g_1$. Suppose that   $(H_1)-(H_2)$ hold, and let $\delta:= \deg(R_g)$. Then we have the following estimates on $|V_g(\fq)|$, the number of $\fq$--rational points of $V_g$:
 \begin{itemize}  
 \item if  $e <  \deg(R_g-g)$ 
\begin{equation*} \label{estimation 2}
\big||V_g(\fq)|-q^{n-1}\big|\leq (q^{3/2}+1)q^{\frac{n+d-4}{2}}((\delta-1)^{n-d}q^{1/2}+6(\delta+2)^{n}),
\end{equation*}
\item if $e = \deg(R_g-g)$
\begin{equation*} \label{estimation 3}
\big||V_g(\fq)|-q^{n-1}\big|\leq (q+1)q^{\frac{n+d-3}{2}}((\delta-1)^{n-d-1}q^{1/2}+6(\delta+2)^{n}),
\end{equation*}
\item if $e > \deg(R_g-g)$
\begin{equation*} \label{estimation 4}
\big||V_g(\fq)|-q^{n-1}\big|\leq  q^{\frac{n-2}{2}}\big(((\delta-1)^{n-d-1}q^{1/2}+6(\delta+2)^{n})q^{(d+1)/2} +(\delta-1)^{n-1}\big).
\end{equation*}
On the other hand, if $g\in \fq$ we have that 
\begin{equation*} \label{estimation 1}
\big||V_g(\fq)|-q^{n-1}\big|\leq (q+1)q^{\frac{n+d-4}{2}}((\delta-1)^{n-d}q^{1/2}+6(\delta+2)^{n}).
\end{equation*}
\end{itemize}
\end{theorem}

The paper is organized as follows. In Section 2 we collect the notions of algebraic geometry we use. In Section 3 we 
study the geometric properties of the \textit{deformed hypersurfaces} $V_g$ and we settle Theorem \ref{teo 1 intro}. Finally, in Section 4 we apply our methodology to obtain estimates and existence results of deformed diagonal equations, generalized Markoff-Hurwitz-type equations and Carlitz's equations.

\section{Basic notions of algebraic geometry}
\label{sec: notions of algebraic geometry}
In this section we collect the basic definitions and facts of
algebraic geometry that we need in the sequel. We use standard
notions and notations which can be found in, e.g., \cite{Kunz85},
\cite{Shafarevich94}.

Let $\K$ be any of the fields $\fq$ or $\cfq$. We denote by $\A^r$
the affine $r$--dimensional space $\cfq{\!}^{r}$ and by $\Pp^r$ the
projective $r$--dimensional space over $\cfq{\!}^{r+1}$. Both spaces
are endowed with their respective Zariski topologies over $\K$, for
which a closed set is the zero locus of a set of polynomials of
$\K[X_1,\ldots, X_r]$, or of a set of homogeneous polynomials of
$\K[X_0,\ldots, X_r]$.

A subset $V\subset \Pp^r$ is a {\em projective variety defined over}
$\K$ (or a projective $\K$--variety for short) if it is the set of
common zeros in $\Pp^r$ of homogeneous polynomials $F_1,\ldots, F_m
\in\K[X_0 ,\ldots, X_r]$. Correspondingly, an {\em affine variety of
	$\A^r$ defined over} $\K$ (or an affine $\K$--variety) is the set of
common zeros in $\A^r$ of polynomials $F_1,\ldots, F_{m} \in
\K[X_1,\ldots, X_r]$. We think a projective or affine $\K$--variety
to be equipped with the induced Zariski topology. We shall denote by
$\{F_1=0,\ldots, F_m=0\}$ or $V(F_1,\ldots,F_m)$ the affine or
projective $\K$--variety consisting of the common zeros of
$F_1,\ldots, F_m$.

In the remaining part of this section, unless otherwise stated, all
results referring to varieties in general should be understood as
valid for both projective and affine varieties.

A $\K$--variety $V$ is {\em irreducible} if it cannot be expressed
as a finite union of proper $\K$--subvarieties of $V$. Further, $V$
is {\em absolutely irreducible} if it is $\cfq$--irreducible as a
$\cfq$--variety. Any $\K$--variety $V$ can be expressed as an
irredundant union $V=\mathcal{C}_1\cup \cdots\cup\mathcal{C}_s$ of
irreducible (absolutely irreducible) $\K$--varieties, unique up to
reordering, called the {\em irreducible} ({\em absolutely
	irreducible}) $\K$--{\em components} of $V$.

For a $\K$--variety $V$ contained in $\Pp^r$ or $\A^r$, its {\em
	defining ideal} $I(V)$ is the set of polynomials of $\K[X_0,\ldots,
X_r]$, or of $\K[X_1,\ldots, X_r]$, vanishing on $V$. The {\em
	coordinate ring} $\K[V]$ of $V$ is the quotient ring
$\K[X_0,\ldots,X_r]/I(V)$ or $\K[X_1,\ldots,X_r]/I(V)$. The {\em
	dimension} $\dim V$ of $V$ is the length $n$ of a longest chain
$V_0\varsubsetneq V_1 \varsubsetneq\cdots \varsubsetneq V_n$ of
nonempty irreducible $\K$--varieties contained in $V$. 
We say that $V$ has {\em pure dimension} $n$ if every irreducible
$\K$--component of $V$ has dimension $n$. A $\K$--variety of $\Pp^r$
or $\A^r$ of pure dimension $r-1$ is called a $\K$--{\em
	hypersurface}. A $\K$--hypersurface of $\Pp^r$ (or $\A^r$) can also
be described as the set of zeros of a single nonzero polynomial of
$\K[X_0,\ldots, X_r]$ (or of $\K[X_1,\ldots, X_r]$).

The {\em degree} $\deg V$ of an irreducible $\K$--variety $V$ is the
maximum of $|V\cap L|$, considering all the linear spaces $L$ of
codimension $\dim V$ such that $|V\cap L|<\infty$. More generally,
following \cite{Heintz83} (see also \cite{Fulton84}), if
$V=\mathcal{C}_1\cup\cdots\cup \mathcal{C}_s$ is the decomposition
of $V$ into irreducible $\K$--components, we define the degree of
$V$ as
$$\deg V:=\sum_{i=1}^s\deg \mathcal{C}_i.$$
The degree of a $\K$--hypersurface $V$ is the degree of a polynomial
of minimal degree defining $V$. 
%

Let $V\subset\A^r$ be a $\K$--variety, $I(V)\subset \K[X_1,\ldots,
X_r]$ its defining ideal and $x$ a point of $V$. The {\em dimension}
$\dim_xV$ {\em of} $V$ {\em at} $x$ is the maximum of the dimensions
of the irreducible $\K$--components of $V$ containing $x$. If
$I(V)=(F_1,\ldots, F_m)$, the {\em tangent space} $\mathcal{T}_xV$
to $V$ at $x$ is the kernel of the Jacobian matrix $(\partial
F_i/\partial X_j)_{1\le i\le m,1\le j\le r}(x)$ of $F_1,\ldots, F_m$
with respect to $X_1,\ldots, X_r$ at $x$. We have
$\dim\mathcal{T}_xV\ge \dim_xV$ (see, e.g., \cite[page
94]{Shafarevich94}). The point $x$ is {\em regular} if
$\dim\mathcal{T}_xV=\dim_xV$; otherwise, $x$ is called {\em
	singular}. The set of singular points of $V$ is the {\em singular
	locus} of $V$; it is a closed $\K$--subvariety of
$V$. A variety is called {\em nonsingular} if its singular locus is
empty. For projective varieties, the concepts of tangent space,
regular and singular point can be defined by considering an affine
neighborhood of the point under consideration.

%
%
\subsection{Rational points}
Let $\Pp^r(\fq)$ be the $r$--dimensional projective space over $\fq$
and $\A^r(\fq)$ the $r$--dimensional $\fq$--vector space $\fq^n$.
For a projective variety $V\subset\Pp^r$ or an affine variety
$V\subset\A^r$, we denote by $V(\fq)$ the set of $\fq$--rational
points of $V$, namely $V(\fq):=V\cap \Pp^r(\fq)$ in the projective
case and $V(\fq):=V\cap \A^r(\fq)$ in the affine case. For an affine
variety $V$ of dimension $n$ and degree $\delta$, we have the
following bound (see, e.g., \cite[Lemma 2.1]{CaMa06}):
\begin{equation*}\label{eq: upper bound -- affine gral}
|V(\fq)|\leq \delta\, q^n.
\end{equation*}
On the other hand, if $V$ is a projective variety of dimension $n$
and degree $\delta$, then we have the following bound (see
\cite[Proposition 12.1]{GhLa02a} or \cite[Proposition 3.1]{CaMa07};
see \cite{LaRo15} for more precise upper bounds):
\begin{equation*}\label{eq: upper bound -- projective gral}
|V(\fq)|\leq \delta\, p_n,
\end{equation*}
where $p_n:=q^n+q^{n-1}+\cdots+q+1=|\Pp^n(\fq)|$.
%
%
\subsection{Complete intersections}\label{subsec: complete intersections}
Elements $F_1,\ldots, F_m$ in $\mathbb{K}[X_1,\ldots,X_r]$ or
$\mathbb{K}[X_0,\ldots,X_r]$ form a \emph{regular sequence} if $F_1$
is nonzero and no $F_i$ is zero or a zero divisor in the quotient
ring $\mathbb{K}[X_1,\ldots,X_r]/ (F_1,\ldots,F_{i-1})$ or
$\mathbb{K}[X_0,\ldots,X_r]/ (F_1,\ldots,F_{i-1})$ for $2\leq i \leq
m$.  In such a case, the (affine or projective) variety $V := V (F_1,\ldots,F_m)$ they define is equidimensional of dimension $r-m$, and is called a \emph{set–theoretic complete intersection}.

For  given  positive  integers $a_1, \ldots,a_r$, we define the weight $\wt({\boldsymbol{X^\alpha}})$ of a monomial $\boldsymbol{X^{\alpha}}:=X_1^{\alpha_1}\cdots X_r^{\alpha_r}$ as $\wt({\boldsymbol{X^\alpha}}):=\sum_{i=1}^r a_i\cdot \alpha_i$. The weight $\wt(f)$ of an arbitrary element $f\in  \mathbb{K}[X_1,\ldots,X_r]$ is the highest weight of all the monomials with nonzero coefficients arising in the dense representation of $f$.  
\begin{lemma} \cite[Lemma 5.4]{MaPePr19} \label{lemma: regular sequences 2}
Let $F_1 \klk F_m \in \K[X_1 \klk X_{r}]$. For an assignment of
posi\-tive integer weights $\wt$ to the variables $X_1 \klk X_{r}$,
denote by $F_1^{\wt} \klk F_m^{\wt}$ the components of highest
weight of $F_1 \klk F_m$. If $F_1^{\wt} \klk F_m^{\wt}$ form a
regular sequence in $\K[X_1 \klk X_{r}]$, then $F_1 \klk F_m$
form a regular sequence in $\K[X_1 \klk X_{r}]$.
\end{lemma}

\section{Deformed hypersurfaces defined by $m_j$--power sums}
In this section we shall develop the extension of the methodology used in  \cite{CaMaPr12}, \cite{CeMaPePr14}, \cite{MaPePr14}, \cite{CeMaPe17} and \cite{MaPePr19} to the more general case of equations given by polynomials evaluated in power-sum polynomials.
Let $d, n,e$ be positive integers such that $1 \leq d\leq n-3$ and let $m_{1},\ldots, m_{d}$ be positive integers with  $2\leq  m_{1}<\cdots<  m_{d}$. We assume that $\mathrm{char}(\fq)$ does not divide $e$ and  $m_{j}$ for all $1\leq j \leq d$. Let $Y_1,\ldots,Y_d$ be indeterminates over $\fq$ and let  $f \in \fq[Y_1,\ldots,Y_d]$.  We consider the weight $\wt :\fq[Y_1,\ldots,Y_d] \longrightarrow \N$
defined by setting $\wt(Y_j) :=m_{j}$ for $1\leq j \leq d$ and denote by $f^{\wt}$ the component
of highest weight of $f$.
Let $\nabla f$ and $\nabla f^{\wt}$ be
 the gradients of $f$ and $f^{\wt}$ respectively.
 Suppose that the following hypotheses hold:
 \begin{enumerate}
 \item[$(H_1)$] $\nabla f\neq 0$ on every point of $ \A^d$.
\item[$(H_2)$] $\nabla f^{\wt}\neq 0$ on every point of $ \A^d$.
 \end{enumerate}
Let  $X_1,\ldots,X_n$ be new indeterminates over $\fq$.  We consider the $m_j$-power sum $P_{m_j}=X_1^{m_{j}}+\cdots+X_n^{m_{j}},\,\,\, 1\leq j \leq d$.
 Finally, let  $g\in \fq$ or let $g\in \fq[X_1,\ldots,X_n]$ be the following polynomial 
 \begin{equation}\label{def: g} g=X_1^e+\cdots + X_n^e+g_1, \end{equation}
 where $g_1\in \fq[X_1,\ldots,X_n]$ and $\deg(g_1)<e$.
%
 Our aim is to estimate the number of $\fq$--rational solutions of the equation
$$f(P_{m_1} \klk P_{m_d})+ g=0.$$
To do this, we consider 
 $R_g\in \fq[X_1,\ldots,X_n]$  the polynomial  \begin{equation}\label{polynomial Rg} R_g:= f(P_{m_1} \klk P_{m_d})+ g.\end{equation} Let $V_g:=V(R_g)\subset \A^n$ be the $\fq$-affine hypersurface defined by $R_g$. We call $V_g$ a \textit{deformed hypersurface} defined by the $m_j$-powers sum with coefficients in $\fq$.
We shall study some facts concerning the geometry of $V_g$. For this purpose, we need to obtain an upper bound of the dimension of $\Sigma_g$,  the singular locus  of $V_g$. Then, for a given $\bf{x}$ $\in \A^n$, we shall consider the following $(d\times n)$--matrix $A(\bf{x})$: 
 \begin{equation} \label{eq:jacobiano de los pi}
    A(\mathbf{x}):=
       \left(
         \begin{array}{ccc}
           \dfrac{\partial P_{m_1}}{\partial X_1}(\mathbf{x}) & \cdots & \dfrac{\partial P_{m_1}}{\partial X_{n}}(\mathbf{x})
         \\
           \vdots & & \vdots
          \\
           \dfrac{\partial P_{m_d}}{\partial X_1}(\mathbf{x})& \cdots & \dfrac{\partial P_{m_d}}{\partial X_{n}}(\mathbf{x})
         \end{array}
       \right).
  \end{equation}
By the chain rule, we deduce that the partial derivatives of $f(\bf P)(\bf X)$ satisfy the following equality for $1\leq j \le n:$
$$\dfrac{\partial f(\bf P)}{\partial X_j}  =
\bigg(\dfrac{\partial f}{\partial
Y_1}\circ\bfs{P}\bigg)\cdot\dfrac{\partial P_{m_1}}{\partial
X_j}+\cdots+\bigg(\dfrac{\partial f}{\partial
Y_d}\circ\bfs{P}\bigg)\cdot\dfrac{\partial P_{m_d}}{\partial X_j},$$
where $\bfs P=(P_{m_1} \klk P_{m_d})$.
Suppose firstly that $g$ is defined as in \eqref{def: g}. For any ${\bf{x}} \in\Sigma_g$,  we have 
$$
\nabla R_g({\bf{x}})=\nabla f (\boldsymbol{P}({\bf{x}}))\cdot A ({\bf{x}}) +  \nabla g(\bf {x})= \bf{0}.
$$
Then  ${\bf{y}}:=\nabla f
 {(\boldsymbol{P}({\bf{x}}))}$ is a  solution of the system 
 \begin{equation}\label{eq: system} A^{t}({\bf{x}})Y^{t}= - \nabla g({\bf {x}})^{t}.
 \end{equation}
 	Observe that for $ \nabla g(\bf {x})= \bf{0}$ $(H_1)$ implies that \eqref{eq: system} has a nonzero solution. Therefore $\rank(A({\bf{x}}))< d$. 	
 	On the other hand, if $ \nabla g(\bf {x})\neq  \bf{0}$, either $\rank(A({\bf{x}}))< d$  or $\rank(A({\bf{x}}))=\rank(M_A({\bf{x}}))=d $,
 	where $M_A({\bf{x}})$ is the augmented matrix of the system \eqref{eq: system}.  
 Hence $\Sigma_g \subset Z_1\cup Z_2$, where $Z_1$ and $Z_2$ are the following sets:
 \begin{equation*}\label{Z_1 and Z_2}\begin{aligned}
 Z_1&=\{{\bf{x}}\in \A^n:\, \rank(A({\bf{x}}))< d \}\\
  Z_2&=\{{\bf{x}}\in \A^n:\, \rank(A({\bf{x}}))=\rank(M_A({\bf{x}})) =d\}.
  \end{aligned}
 \end{equation*} 
 Suppose now that $g\in \fq$ and let ${\bf{x}} \in \Sigma_g$. We have that 
 $$\nabla R_g({\bf{x}})=\nabla f (\boldsymbol{P}({\bf{x}}))\cdot A ({\bf{x}})= \bf{0}.
 $$
 From $(H_1)$,  we have that ${\bf{y}}:=\nabla f
 {(\boldsymbol{P}({\bf{x}}))}$ is a nonzero solution of the homogeneous system 
 $$ A^{t}({\bf{x}})Y^{t}= 0.
 $$
We conclude that $\rank(A({\bf{x}})) <d$. Therefore ${\bf{x}}\in Z_1$ and $\Sigma_g \subset Z_1$.

 We next obtain an upper bound of the dimension of $Z_1$. 
 In what follows, $VDM(X_1,\ldots,X_t)$ stands for the Vandermonde matrix in the variables $X_1,\ldots,X_t$.
 \begin{proposition}\label{lemma: dimension de Z1} $Z_1$ has dimension at most $d-1$.
 \end{proposition}
 \begin{proof}
 
Observe that  $$Z_1=\bigcup _{j=0}^{d-1} V_j,$$
where $V_j:=\{\mathbf{x}\in \A^n\,:\, \rank( A(\mathbf{x}))=j\}.$ Since $\mathrm{char}(\fq)$ does not divide  $m_{k}$ for all $1\leq k \leq d$ then
\begin{equation} \label{eq:matriz A}
    A(\mathbf{X}):=
       \left(
         \begin{array}{ccc}
           m_{1}X_1^{m_{1}-1} & \cdots &  m_{1}X_n^{m_{1}-1}
         \\
           \vdots & & \vdots
          \\
           m_{d}X_1^{m_{d}-1} & \cdots &  m_{d}X_n^{m_{d}-1}
         \end{array}
       \right).
  \end{equation}
 Fix $j$  for $0\leq j \leq d-1$.  We shall prove that the dimension of $V_j$ is at most $d-1$.  Let ${\bf{x}} \in V_j$, so $\rank(A({\bf{x}}))=j$. Thus, there exists a  $(j\times j)$--submatrix of $A({\bf{x}})$ of $\rank=j$. Suppose without loss of generality that this submatrix consists of the first $j$ rows and the first $j$ columns de $A({\bf{x}})$.
 Let $C_1,\ldots,C_n$ denote the columns of $A(\mathbf{X})$ and let $A(C_{1},\ldots,C_{j}, C_k)$  be the $((j+1)\times (j+1))$--submatrix of $A(\bf{X})$ consisting of the entries of the first $j+1$ rows  and the columns $C_1,\ldots,C_{j},C_k$, $j+1\leq k  \leq n$. By \cite[Proposition 5]{BGHM}, ${\bf{x}}$ belongs to the set of common zeros of the $(n-j)$ $\fq$--equations 
 $$\det A(C_{1},\ldots,C_{j}, C_k)=0,\,\,  j+1 \leq k \leq n.$$ 
Now, from \cite[Theorem 2.1]{DM02}: 
  $$\det A(C_{1},\ldots,C_{j}, C_k)(\mathbf{X})=\text{VDM}(X_1,\ldots,X_{j})\Pi_{l=1}^{j}X_l^{m_{1-1}}(X_k-X_l)P(X_k),\,\,  j+1 \leq k \leq n,$$
  for some $P\in \fq[T]$. Since the principal minor of $A({\bf {x}})$ is nonzero, we deduce that $x_l\neq 0$, $1\leq l \leq j$, and $x_m\neq x_n$, $1\leq m<n\leq j$. Therefore ${\bf{x}} \in V(Q_{j+1},\ldots,Q_{n})$, where $Q_k:=\Pi_{l=1}^{j}(X_k-X_l)P(X_k)$.
  We claim that $Q_{j+1},\ldots,Q_n$ form a regular sequence of $\fq[X_1,\ldots,X_n]$. Indeed, consider the graded lexicographic  order of $\fq[X_1,\ldots,X_n]$ with $X_n>X_{n-1}>\cdots >X_1$. With this order we have that $Lt(Q_k)=X_k^j\cdot P_k^j$, where $Lt(Q_k)$ denotes the leading terms of the polynomials $Q_k$. Thus $Lt(Q_k)$, $j+1\leq k \leq n,$ are relatively prime and they form a
Gröbner basis of the ideal $J$ that they generate (see, e.g., \cite[\S 2.9, Proposition 4]{CLO92}). Hence,
the initial of the ideal $J$ is generated by $Lt(Q_{j+1}),\ldots,Lt(Q_n)$, which form a regular sequence
of $\fq[X_1,\ldots,X_n]$. Therefore, by \cite[Proposition 15.15]{Eisenbud95}, the polynomials $Q_{j+1},\ldots,Q_n$ form
a regular sequence of $\fq[X_1,\ldots,X_n]$. We conclude that $V(Q_{j+1},\ldots,Q_{n})$ is a set complete intersection of $\A^n$ of dimension $j$. 
  
Observe finally that, given ${\bf{x}} \in V_j$, there exists a $(d \times n)$-matrix $A'({\bf{X}})$ obtained by rearranging  the columns of $A({\bf X})$ with nonzero principal minor. Therefore, by the  same arguments as above, $V_j$ is included in  an union of $\fq$-varieties of dimension $j$. Thus, $Z_1$ has dimension at most $d-1$. 
\end{proof} 
Now, we obtain an upper bound of the dimension of $Z_2$. First we need the following remark.
\begin{remark}\label{remark derivada de g}
 For all ${\bf x}\in Z_2$, $\nabla(g({\bf x}))\neq 0$. Indeed, $\nabla(g({\bf x}))=0$ implies that the  system \eqref{eq: system} is homogeneous. Since it has a nonzero solution then $\rank( A({\bf{x}}))<d$,  which contradicts ${\bf x}\in Z_2$.
 \end{remark}
\begin{proposition} \label{lemma: dimension de Z2} The dimension of $Z_2$ is at most $d$.
\end{proposition}

\begin{proof}
Let ${\bf{x}} \in Z_2$, so $\rank(A({\bf{x}}))= d.$ Thus, there exists a $(d\times d)$--submatrix $B({\bf{x}})$  of $A({\bf{x}})$  of $\rank=d$. Suppose that this submatrix consists of the first $d$ rows and columns of $A({\bf{x}})$.  Then $ {\bf{x}}\in B:=\{{\bf x} \in \A^n:\,\, \det B({\bf{x}})\neq0,\,\,\rank(M_A({\bf{x}}))=d\}$.
Taking into account that
 \begin{equation} \label{eq: det matriz A p}
  \det B(\mathbf{x}):= \det
       \left(
         \begin{array}{ccc}
           m_{1}x_1^{m_{1}-1} & \cdots &  m_{1}x_d^{m_{1}-1}
         \\
           \vdots & & \vdots
          \\
           m_{d}x_1^{m_{d}-1} & \cdots &  m_{d}x_d^{m_{d}-1}
         \end{array}
       \right)\neq 0
  \end{equation}
and   \cite[Proposition 5]{BGHM}, we have that $B \subset V(F_{d+1},\ldots,F_{n})$, where $F_j\in \fq[X_1,\ldots,X_n]$, $d+1\leq j\leq n$, are the polynomials:
 
  \begin{equation*}F_j:=\det\left(
         \begin{array}{cccc}
           m_1 X_1^{m_{1}-1} & \cdots &  m_{1}X_d^{m_{1}-1}&m_{1}X_j^{m_{1}-1}
         \\
           \vdots & & \vdots & \vdots
          \\
           m_{d}X_1^{m_{d}-1} & \cdots &  m_{d}X_d^{m_{d}-1}&m_{d}X_j^{m_{d}-1}\\
           & & & \\
           -\frac{\partial g}{\partial X_1}&\cdots & -\frac{\partial g}{\partial X_d}&-\frac{\partial g}{\partial X_j}
         \end{array}
       \right).
       \end{equation*}
By the definition of $g$, the component $F_j^{\deg(F_j)}$ of highest degree of $F_j$ is the following polynomial: 
        \begin{equation*}F_j^{\deg(F_j)}:=\det\left(
       \begin{array}{cccc}
       m_1 X_1^{m_{1}-1} & \cdots &  m_{1}X_d^{m_{1}-1}&m_{1}X_j^{m_{1}-1}
       \\
       \vdots & & \vdots & \vdots
       \\
       m_{d}X_1^{m_{d}-1} & \cdots &  m_{d}X_d^{m_{d}-1}&m_{d}X_j^{m_{d}-1}\\
       & & & \\
       -eX_1^{e-1} & \cdots &  -eX_d^{e-1}&-eX_j^{e-1}
       \end{array}
       \right).
       \end{equation*}
        From \cite[Theorem 2.1]{DM02} we have that
       $$F_j^{\deg(F_j)}=\text{VDM}(X_1,\ldots,X_d)\Pi_{k=1}^{d}X_k^{{\min\{m_1,e\}-1}}(X_j-X_k)P(X_j)\,\, , d+1 \leq j \leq n,$$
       for some $P\in \fq[T]$. From \eqref{eq: det matriz A p} we deduce that $x_k\neq 0$, $1\leq k \leq d$, and $x_i\neq x_k$, $1\leq i<k\leq d$. Therefore $V(F_{d+1}^{\deg F_{d+1}},\ldots,F_j^{\deg F_j})= V(Q_{d+1},\ldots,Q_{j})$, where $d+1\leq j\leq n$ and   $Q_j:=\Pi_{k=1}^{d}(X_j-X_k)P(X_j)$.
       The same reasoning as in the proof of Proposition \ref{lemma: dimension de Z1} shows that $Q_{d+1},\ldots,Q_n$ form a regular sequence of $\fq[X_1,\ldots,X_n]$. Hence, $\dim V(F_{d+1}^{\deg F_{d+1}},\ldots,F_j^{\deg F_j})=n-j+d$ for $ d+1 \leq j \leq n$. Therefore,  $F_{d+1}^{\deg F_{d+1}},\ldots,F_n^{\deg F_n}$ form a regular sequence of $\fq[X_1,\ldots,X_n]$ and, by Lemma \ref{lemma: regular sequences 2}, $F_{d+1},\ldots,F_n$ form a regular sequence of $\fq[X_1,\ldots,X_n]$. This proves $B$ has dimension at most $d$.

       Finally, observe that, given ${\bf{x}} \in Z_2$, there exists a $(d \times n)$-matrix $A'({\bf{X}})$ obtained from $A({\bf X})$ by reordering its columns, with the condition that the principal minor of $A'({\bf{X}})$ is nonzero. Therefore, we conclude that $Z_2$ is included in  an union of $\fq$-varieties of dimension at most $d$. 
\end{proof}

From Proposition \ref{lemma: dimension de Z1} and Proposition \ref{lemma: dimension de Z2} we obtain the following result.
\begin{theorem} \label{theo: dimension del lugar singular en el affin}
Let $1 \leq d \leq n-3$. Let $f\in \fq[Y_1, \dots, Y_d]$ and $g \in \fq[X_1, \dots, X_n]$ defined as in \eqref{def: g} such that $(H_1)-(H_2)$ hold and $R_g$ is defined as in \eqref{polynomial Rg}. The dimension of the singular locus $\Sigma_g$ of $V_g$ has dimension at most $d$. On the other hand, if $g\in \fq$, the dimension of the singular locus $\Sigma_g$ of $V_g$ has dimension at most $d-1$.
\end{theorem}  
We shall also need information concerning the behaviour of $V_g$ at ``infinity''. For this purpose, we consider the projective closure $\mathrm{pcl}(V_g)\subset \Pp^n$. It is well known that $\mathrm{pcl}(V_g)$ is the $\fq$-hypersurface of $\Pp^n$ defined by the homogenization $R_g^h\in \fq[X_0,\ldots,X_n]$ of the polynomial $R_g$ (see, e.g.,\cite[\S I.5, Exercise 6]{Kunz85}).

Let $R_g^{\deg(R_g)}$ be the component of highest degree of $R_g$. We shall express $R_g^{\deg(R_g)}$ in terms of the component $f^{\wt}$ of highest weight of $f$. Let $Y_1^{j_1}\cdots Y_d^{j_d}$ be a monomial
 with nonzero coefficient arising in the dense representation of $f$. Then its weight 
$$\wt(Y_1^{j_1}\cdots Y_d^{j_d})=m_{1} j_1 + \cdots +m_{d} j_d$$
is equal to the degree of the corresponding monomial
$P_{m_1}^{j_1} \cdots P_{m_d}^{j_d} $ of $R_g$. From these arguments we deduce the following lemma.
\begin{lemma}\label{lemma: Rdeg}  Suppose that $g$ is defined as in  \eqref{def: g}.Then,
$$R_g^{\deg(R_g)}=\begin{cases}
f^{\wt}(P_{m_1} \klk P_{m_d}) & e< \wt(f)\\
f^{\wt}(P_{m_1} \klk P_{m_d})+X_1^e+ \cdots +X_n^e & e=\wt(f)\\
X_1^e+ \cdots +X_n^e & e>\wt(f)
\end{cases}$$
  On the other hand, if $g\in \fq$, we have that $R_g^{\deg(R_g)}=f^{\wt}(P_{m_1} \klk P_{m_d})$. 
\end{lemma}

\begin{proposition} \label{prop: dimension del lugar singular en el infinito}
Let $1\leq d \leq n-3$.  Suppose that $g$ is defined as in \eqref{def: g}. 
\begin{itemize}
\item if $e< \wt(f)$ then  $\mathrm{pcl}(V_g)$ has singular locus at infinity of dimension at most $d-2,$
\item If $e= \wt(f)$ $\mathrm{pcl}(V_g)$ has singular locus at infinity of dimension at most  $d-1$,
\item if $e>
 \wt(f)$ then $\mathrm{pcl}(V_g)$  has no singular points at infinity.
\end{itemize}
   On the other hand, if $g\in \fq$ then  $\mathrm{pcl}(V_g)$ has singular locus at infinity of dimension at most $d-2$.

\end{proposition}
\begin{proof}
 Let $\Sigma_g^{\infty}\subset \Pp^n$ denote the singular locus of $\mathrm{pcl}(V_g)$ at infinity; namely, the set of singular points of $\mathrm{pcl}(V_g)$ lying in the hyperplane $\{X_0=0\}.$ 
We have that $R_g^h(0,X_1,\ldots,X_n)=R_g^{\deg (R_g)}(X_1,\ldots,X_n)$. 
%
Suppose first that $e<\wt(f)$. From Lemma \ref{lemma: Rdeg}, $R_g^{\deg (R_g)}=f^{\wt}(P_{m_1},\ldots,P_{m_d})$. Thus, any point ${\bf{x}}=(0:x_1: \dots:x_n) \in  \Sigma_g^{\infty}$ satisfies the equations:

\begin{equation}\label{eq1}
f^{\wt}({\bf{P}})=0\,\,\,\, , \frac{{\partial f^{\wt}}({\bf{P}})}{\partial {X_j}}=0,\,\,\,\,1\leq j \leq n
.\end{equation} 
From $(H_2)$ we have that the homogeneous system
$ A^t({\bf{x}}) Y^t={ \bf{0}},$
where $A({\bf{x}})$ is defined as in \eqref{eq:jacobiano de los pi}, has a nonzero solution $\nabla f^{ \wt} (\boldsymbol{P}({\bf{x}}))$. We conclude that ${\bf{x}}\in Z_1$. Thus, we deduce from Proposition \ref{lemma: dimension de Z1} that the set of solutions of \eqref{eq1} is an affine cone of $\A^n$ of dimension at most $d-1$ and, hence, a proyective variety of $\Pp^{n-1}$ of dimension at most $d-2$. Therefore,
the set of singular points of $\mathrm{pcl}(V_g)$ lying in the hyperplane $\{X_0=0\}$ has dimension at most $d-2$. 

Suppose now that $e=\wt
(f).$ From Lemma \ref{lemma: Rdeg} $R_g^{\deg(R_g)}=f^{\wt}(P_{m_1} \klk P_{m_d})+X_1^e+\cdots +X_n^e$. Then, by $(H_2)$ and the same arguments of the proof of Proposition \ref{lemma: dimension de Z2}, we deduce that the variety defined by the equations:
$$f^{\wt}({\bf{P}})+ X_1^e+\cdots +X_n^e=0,\,\,\,\, \frac{\partial{f^{\wt}({\bf{P}})}}{\partial {X_j}} + eX_j^{e-1} =0,\,\,\,\,1\leq j \leq n,
$$
is an affine cone of $A^n$ of dimension at most $d$. Therefore, 
the set of singular points of $\mathrm{pcl}(V_g)$ lying in the hyperplane $\{X_0=0\}$ has dimension at most $d-1$. 

Finally, if $e> \wt(f)$, observe that $R_g^{\deg (R_g)}=  X_1^e+\cdots +X_n^e$, from where it is easy to see that $\mathrm{pcl}(V_g)$
has no singular points at infinity.

Suppose now that $g\in \fq$.  We see that any point ${\bf{x}}=(0:x_1: \dots:x_n) \in  \Sigma_g^{\infty}$ satisfies the equations defined  in \eqref{eq1}. Hence, by $(H_2)$ and the same arguments of the proof of the case $e <\wt(f)$, we deduce that the set of singular points of $\mathrm{pcl}(V_g)$ lying in the hyperplane $\{X_0=0\}$ has dimension at most $d-2$.
\end{proof}
From Theorem \ref{theo: dimension del lugar singular en el affin} and Proposition \ref{prop: dimension del lugar singular en el infinito} we obtain the following result.
\begin{theorem} \label{prop: dimension del lugar singular de Vg}
Let $1 \leq d \leq n-3$. Let $f\in \fq[Y_1, \dots, Y_d]$, $g \in \fq[X_1, \dots, X_n]$ defined as in \eqref{def: g} such that $(H_1)-(H_2)$ hold and $R_g$  defined as in \eqref{polynomial Rg}. 
Then $\mathrm{pcl}(V_g)$ has singular locus of dimension at most  $d$. On the other hand, if $g\in \fq$ then $\mathrm{pcl}(V_g)$ has singular locus of dimension at most  $d-1$.
\end{theorem}

\begin{corollary}\label{teo: V es absolutamente irreducible}
The hypersurface $V_g$ is absolutely irreducible, when $g$ is defined as in \eqref{def: g} or $g\in \fq$.
\end{corollary}
\begin{proof} Observe that $V_g$ is absolutely irreducible if and only if $\mathrm{pcl}(V_g)$ is absolutely irreducible (see, e.g., \cite[Chapter I, Proposition 5.17]{Kunz85}).
Suppose that $\mathrm{pcl}(V_g)$
is not absolutely irreducible. Then it has a nontrivial
decomposition into absolutely irreducible components
$$\mathrm{pcl}(V_g)=\mathcal{C}_1\cup\cdots\cup \mathcal{C}_s,$$
where $\mathcal{C}_1\klk\mathcal{C}_s$ are projective hypersurfaces
of $\Pp^{n}$. Since $\mathcal{C}_i\cap\mathcal{C}_j\neq\emptyset$
and $\mathcal{C}_i$, $\mathcal{C}_j$ are absolutely irreducible, then $\dim (\mathcal{C}_i\cap \mathcal{C}_j)=n-2$.

Denote by $\Sigma_g^{h}$ the singular locus of $\mathrm{pcl}(V_g)$. From 
Theorem \ref{prop: dimension del lugar singular de Vg} we have that
$\dim\Sigma_g^h\le d$. On the other hand, we have
$\mathcal{C}_i\cap \mathcal{C}_j\subset\Sigma_g^h$ for any $i\not=
j$, which implies $\dim \Sigma_g^h\geq n-2$. This contradicts the
assertion $\dim\Sigma_g^{h}\leq d$, since  $d\leq n-3$ by
hypothesis. 
\end{proof}

\subsection{Estimates on the number of $\fq$--rational points of  deformed hypersurfaces}
Let $d, n, e$ be positive integers such that $1 \leq d\leq n-3$ . In this section we shall estimate the number of $\fq$-rational points of $V_g:=V(R_g)\subset \A^n$, where $R_g$ is defined as in \eqref{polynomial Rg}, thus proving Theorem \ref{teo 1 intro}.

In what follows, we shall use an estimate on the number of
$\fq$--rational points of a projective hypersurface due to S. Ghorpade
and G. Lachaud (\cite{GhLa02a}; see also \cite{GhLa02}). In \cite[Theorem 6.1]{GhLa02a}, the authors prove that,
for an absolutely irreducible
$\fq$--hypersurface $V\subset \Pp^{m+1}$ of degree $d\geq 2$ and  singular locus of dimension at most $s\geq 0$, the number $|V(\fq)|$ of $\fq$--rational points of $V$ satisfies the estimate:
\begin{equation}\label{eq: estimacion de Ghorpade Lachaud}
\big||V(\fq)|-p_m\big|\leq b_{m-s-1,d}'\, \, q^{\frac{m+s+1}{2}}+C_{s,m}q^{\frac{m+s}{2}},
\end{equation}
where $p_m:=q^m+q^{m-1}+\cdots+1$, $b_{m,d}'$ is the  $m$--th primitive Betti number of
any nonsingular hypersurface of $\Pp^{m+1}$ of degree $d$ and $C_s(V):=\sum_{i=m-1}^{m-1+s}b_{i,\ell}(V)+\varepsilon_i$, where
$b_{i,\ell}(V)$ denotes the $i$--th $\ell$--adic Betti number of $V$ for a prime $\ell$ different from $p:=\mathrm{char}(\fq)$ and
$\varepsilon_i:=1$ for even $i$ and $\varepsilon_i:=0$ for odd $i$. In \cite{CaMaPr12}, the authors
combine the Katz inequality \cite[Theorem 3]{Katz01} with the
Adolphson--Sperber bound for hypersurfaces \cite[Theorem
5.27]{AdSp88} to obtain the following upper bound, which is
slightly better than that for an arbitrary complete intersection
(\cite[Proposition 5.1]{GhLa02a}):
\begin{equation}
\label{eq:upper bound sums betti numbers} C_{s,m}(V) \le
6(d+2)^{m+1}.
\end{equation}
On the other hand, according
to \cite[Theorem 4.1 and Example 4.3]{GhLa02a}, one has the
following upper bound:
\begin{equation}
\label{eq:upper bound betti number} b_{m,d}' \le
\frac{d-1}{d}\big((d-1)^{m+1}-(-1)^{m+1}\big)\le (d-1)^{m+1}.
\end{equation}

Suppose that $g$ defined as in \eqref{def: g}. From Theorem \ref{prop: dimension del lugar singular de Vg} and Corollary \ref{teo: V es absolutamente irreducible}, we   know that $\mathrm{pcl}(V_g)$ is an absolutely irreducible hypersurface of degree $\delta=\deg(R_g)$ and singular locus of dimension at most $d$.
Hence, from \eqref{eq: estimacion de Ghorpade Lachaud}, \eqref{eq:upper bound sums betti numbers} and \eqref{eq:upper bound betti number} we have that:
\begin{equation}\label{eq: estimacion pcl(v) case 2}
\big||\mathrm{pcl}(V_g)(\fq)|-p_{n-1}\big|\leq (\delta
-1)^{n-d-1}q^{\frac{n+d}{2}}+6(\delta+2)^{n} q^{\frac{n+d-1}{2}}.
\end{equation}

Now, we estimate the number of $\fq$--rational points of $V_g^{\infty}= V(\mathrm{pcl}(V_g))\cap \{X_0=0\}$. 
 Note that $V_g^{\infty}=V(R_g^{\deg R_g})$ is a hypersurface of $\Pp^{n-1}$.  
 Suppose that $e < \wt(f)$. The same arguments as in the proof of Proposition \ref{prop: dimension del lugar singular en el infinito} shows that the dimension of the singular locus of $V_g^{\infty}$ is at most $d-2$. Then, taking into account \eqref{eq: estimacion de Ghorpade Lachaud}, the following estimate holds:
\begin{equation}\label{eq: estimacion V en el infinito case 2}
\big||V_g^{\infty}(\fq)|-p_{n-2}\big|\leq (\delta
-1)^{n-d}q^{\frac{n+d-3}{2}}+6(\delta+2)^{n-1} q^{\frac{n+d-4}{2}}.
\end{equation}
Observe that $|V_g(\fq)|=|\mathrm{pcl}(V_g)(\fq)|-|V_g^{\infty}(\fq)|$. From \eqref{eq: estimacion pcl(v) case 2} and \eqref{eq: estimacion V en el infinito case 2}, we have: 
\begin{align*}
\big||V_g(\fq)|-q^{n-1}\big|\leq & \big||\mathrm{pcl}(V_g)(\fq)|-p_{n-1}\big|+\big||V_g^{\infty}(\fq)|-p_{n-2}\big|\\
\leq &(\delta-1)^{n-d-1}q^{\frac{n+d}{2}} + 6 (\delta+2)^{n}q^{\frac{n+d-1}{2}}\\
& +(\delta-1)^{n-d}q^{\frac{n+d-3}{2}}+ 6 (\delta+2)^{n-1}q^{\frac{n+d-4}{2}}\\
\leq & (q^{3/2}+1)q^{\frac{n+d-4}{2}}((\delta-1)^{n-d}q^{1/2}+6(\delta+2)^{n}).
\end{align*}

Suppose now that $e=\wt(f)$. Also by the same arguments as in the proof  of Proposition \ref{prop: dimension del lugar singular en el infinito} the dimension of the singular locus of $V_g^{\infty}$ is at most $d-1$. 
%
Then, taking into account \eqref{eq: estimacion de Ghorpade Lachaud}, the following estimate holds:
\begin{equation}\label{eq: estimacion V en el infinito case 3}
\big||V_g^{\infty}(\fq)|-p_{n-2}\big|\leq (\delta
-1)^{n-d-1}q^{\frac{n+d-2}{2}}+6(\delta+2)^{n-1} q^{\frac{n+d-3}{2}}.
\end{equation}
From \eqref{eq: estimacion pcl(v) case 2} and \eqref{eq: estimacion V en el infinito case 3}, we obtain that 
\begin{align*}
\big||V_g(\fq)|-q^{n-1}\big|\leq & \big||\mathrm{pcl}(V_g)(\fq)|-p_{n-1}\big|+\big||V_g^{\infty}(\fq)|-p_{n-2}\big|\\
\leq &(\delta-1)^{n-d-1}q^{\frac{n+d}{2}} + 6 (\delta+2)^{n}q^{\frac{n+d-1}{2}}\\
& +(\delta-1)^{n-d-1}q^{\frac{n+d-2}{2}}+ 6 (\delta+2)^{n-1}q^{\frac{n+d-3}{2}}\\
\leq & (q+1)q^{\frac{n+d-3}{2}}((\delta-1)^{n-d-1}q^{1/2}+6(\delta+2)^{n}).
\end{align*}
Finally, if $e>\wt(f)$, and again from the proof of Proposition \ref{prop: dimension del lugar singular en el infinito}, we have that  $V_g^{\infty}$ is a nonsingular variety. We can apply the following result due to P. Deligne (see, e.g., \cite{De74}): for a nonsingular ideal-theoretic complete intersection $V \subset \Pp^n$ defined over $\fq$, of dimension $r$ and multidegree ${\bf{d}}=(d_1, \ldots, d_n)$, the following estimate holds:
\begin{equation}\label{estimacion Deligne}
\big||V(\fq)|-p_r| \leq b'_r(n,{\bf{d}})q^{r/2},
\end{equation}	
	where $ b'_r(n,{\bf{d}})$ is the rth-primitive Betti number of any nonsingular complete intersection of $\Pp^n$ of dimension $r$ and multidegree ${\bf{d}}$.

	 Thus
	\begin{equation}\label{estimation V infty e mayor wtf}
	\big||V_g^{\infty}(\fq)|-p_{n-2}\big| \leq (\delta-1)^{n-1} q^{(n-2)/2}.
	\end{equation}
	
	From \eqref{eq: estimacion pcl(v) case 2} and \eqref{estimation V infty e mayor wtf}, we conclude that 
\begin{align*}
\big||V_g(\fq)|-q^{n-1}\big|\leq & \big||\mathrm{pcl}(V_g)(\fq)|-p_{n-1}\big|+\big||V_g^{\infty}(\fq)|-p_{n-2}\big|\\
\leq &(\delta-1)^{n-d-1}q^{\frac{n+d}{2}} + 6 (\delta+2)^{n}q^{\frac{n+d-1}{2}}\\
& +(\delta-1)^{n-1} q^{(n-2)/2}\\
\leq & q^{\frac{n-2}{2}}\big(((\delta-1)^{n-d-1}q^{1/2}+6(\delta+2)^{n})q^{(d+1)/2} +(\delta-1)^{n-1}\big).
\end{align*}

Now, if $g\in \fq$, observe that $\mathrm{pcl}(V_g)\subset \Pp^n$ has degree $\delta=\deg(R_g)$ and, from Theorem \ref{prop: dimension del lugar singular de Vg}, its singular locus has dimension at most $d-1$.
Hence, from \eqref{eq: estimacion de Ghorpade Lachaud}, we obtain:
\begin{equation}\label{eq: estimacion pcl(v)}
\big||\mathrm{pcl}(V_g)(\fq)|-p_{n-1}\big|\leq (\delta
-1)^{n-d}q^{\frac{n+d-1}{2}}+6(\delta+2)^{n} q^{\frac{n+d-2}{2}}.
\end{equation}
On the other hand, $V_g^{\infty}=V(R_g^{\deg R_g})$ is a hypersurface of $\Pp^{n-1}$. As before, the same arguments of Proposition \ref{prop: dimension del lugar singular en el infinito} give us that the dimension of the singular locus of $V_g^{\infty}$ is at most $d-2$. Then taking into account \eqref{eq: estimacion de Ghorpade Lachaud}, the following estimate holds:
\begin{equation}\label{eq: estimacion V en el infinito}
\big||V_g^{\infty}(\fq)|-p_{n-2}\big|\leq (\delta
-1)^{n-d}q^{\frac{n+d-3}{2}}+6(\delta+2)^{n-1} q^{\frac{n+d-4}{2}}.
\end{equation}
From \eqref{eq: estimacion pcl(v)} and \eqref{eq: estimacion V en el infinito}, we conclude that 
\begin{align*}
\big||V_g(\fq)|-q^{n-1}\big|\leq & \big||\mathrm{pcl}(V_g)(\fq)|-p_{n-1}\big|+\big||V_g^{\infty}(\fq)|-p_{n-2}\big|\\
\leq &(\delta-1)^{n-d}q^{\frac{n+d-1}{2}} + 6 (\delta+2)^{n}q^{\frac{n+d-2}{2}}\\
& +(\delta-1)^{n-d}q^{\frac{n+d-3}{2}}+ 6 (\delta+2)^{n-1}q^{\frac{n+d-4}{2}}\\
\leq & (q+1)q^{\frac{n+d-4}{2}}((\delta-1)^{n-d}q^{1/2}+6(\delta+2)^{n}).
\end{align*}

All this previous discussion settles Theorem \ref{teo 1 intro}
%
%
%
%
%
\begin{remark}\label{remark: g1=0} We can provide another estimate for the case when $g_1$ of the definition \eqref{def: g} is identically zero and $R_g$ is an homogeneous polynomial. In this case,
  $V_g \subset \Pp^{n-1}$ is also a projective variety with singular locus of dimension at most $d-1$. Indeed, the same arguments as in the proof of Theorem \ref{theo: dimension del lugar singular en el affin} imply  that the set of ${\bf{x}}\in \A^n$ such that $\nabla R_g({\bf{x}})=0$ defines an affine cone of dimension at most $d$. Hence, the set of ${\bf{x}}\in \Pp^{n-1}$ for which $\nabla R_g({\bf{x}})=0$ has dimension at most $d-1$.  Thus, from \eqref{eq: estimacion de Ghorpade Lachaud}, we have that 
\begin{equation*} 
	\big|\overline{N}_g -p_{n-2}\big|\leq (\delta-1)^{n-d-1}q^{(n+d-2)/2}+6(\delta+2)^{n} q^{(n+d-3)/2},
\end{equation*}
where $\overline{N}_g$ denotes the number of $\fq$--rational projective points of $V_g$. Since $|V_g(\fq)|=\overline{N}_g(q-1)+1$ we conclude that 
\begin{equation*} 
	\big||V_g(\fq)|-q^{n-1}\big|\leq (q-1)\Big ((\delta-1)^{n-d-1}q^{(n+d-2)/2}+6(\delta+2)^{n} q^{(n+d-3)/2}\Big).
\end{equation*}
Note that the order of the error terms in either case $g_1=0$ or $g_1\neq 0$  is the same.
\end{remark}
\begin{remark} It is easy to prove that if $g=0$ and $R_0$ is an homogeneous polynomial then the singular locus of  $V_0\subset \Pp^{n-1}$ has dimension at most $d-2$. Hence we have the following estimate:
\begin{equation*} 
\big||V_0(\fq)|-q^{n-1}\big|\leq (q-1)\Big ((\delta-1)^{n-d-1}q^{(n+d-3)/2}+6(\delta+2)^{n-1} q^{(n+d-4)/2}\Big).
\end{equation*}	
\end{remark}
\subsection{$f$ is a linear polynomial and $n\geq 3$}
For  a linear polynomial $f$ we can obtain  better results. Indeed,  we can  improve Theorem \ref{teo 1 intro} by studying the geometric properties of the deformed hypersurfaces in more detail.
Fix $n\geq 3$ and let $f=b_1Y_1+\dots +b_dY_d+ a \in \fq[Y_1,\ldots,Y_d]$ be a  nonzero linear polynomial.
In the following result we obtain an upper bound of the dimension of $\Sigma_g$ in this case.
\begin{proposition} \label{prop: singular locus f linear}
	$V_g\subset \A^n$ is nonsingular or $\dim\Sigma_g=0$ .
\end{proposition}
\begin{proof} Suppose  first that $g$ is defined as in \eqref{def: g}
	and let ${\bf{x}} \in \Sigma_g$.  Then
	$$
	\nabla R_g({\bf{x}})=(b_1, \dots,b_d)\cdot A({\bf{x}}) + e\cdot(x_1^{e-1},\ldots,x_n^{e-1})+ \nabla g_1(\bf {x})= \bf{0},
		$$
	where $A({\bf{x}})$ is defined in (\ref{eq:matriz A}). Thus we have the following $n$ equations:  $$Q_j:=b_1m_1X_j^{m_1-1}+\dots+b_dm_dX_j^{m_d-1}+ eX_j^{e-1}+\frac{\partial g_1}{\partial X_j}=0, \,\, 1\leq j\leq n.$$
	Consider the graded lexicographic order of $\fq[X_1 \klk X_n]$ with $X_n >X_{n-1}>\cdots >X_1$. 
 For each $1\leq j \leq n$, $Lt(Q_j)$ satisfies:
 \begin{itemize}
 \item $Lt(Q_j)=b_im_iX_j^{m_i-1}$ if $e <m_i$,
 \item $Lt(Q_j)=b_im_iX_j^{m_i-1}+ eX_j^{e-1}$ if $e=m_i$,
 \item $Lt(Q_j)=e X_j^{e-1}$ if $e>m_i$,
 \end{itemize}
   where $m_i:=\max\{m_k,\,  b_k\neq 0, \, 1\leq k \leq n\}$. With this monomial order, the leading terms are relatively prime and thus they form a Gröbner basis of the ideal $J$ that they generate. Hence, the initial of the ideal $J$ is generated by $Lt(Q_1),\ldots,Lt(Q_n)$, which form a regular sequence of $\fq[X_1 \klk X_n]$. Therefore, by  \cite[ Proposition 15.15]{Eisenbud95} the polynomials $Q_{1} \klk Q_n$ form a regular sequence of $\fq[X_1 \klk X_n]$. Thus, we deduce that $\Sigma_g$ has dimension at most $0$.
   
   The case for $g\in \fq$ follows analogously.	
\end{proof}
\begin{corollary}\label{singular locus plc(Vg) at infinity with f linear} $\mathrm{pcl}(V_g)\subset \Pp^n$ has no singular points at infinity.
\end{corollary}

\begin{proof}
	Suppose that $g$ is defined as in \eqref{def: g} and consider $\Sigma_g^{\infty}\subset \Pp^n$. From Lemma \ref{lemma: Rdeg}, we have that: 
\begin{itemize}
\item $ R_g^h(0,X_1,\ldots,X_n)= f^{\wt}(P_{m_1}, \dots, P_{m_d}), $ if $e < \wt(f)$,
\item $R_g^h(0,X_1,\ldots,X_n)= f^{\wt}(P_{m_1}, \dots, P_{m_d})+ X_1^e+\dots+X_n^e,$ if  $e=\wt(f)$,
\item  $R_g^h(0,X_1,\ldots,X_n)= X_1^e+\dots+X_n^e$ if if  $e >\wt(f).$
\end{itemize}	 
On the other hand, if $g\in \fq$ then $R_g^h(0,X_1,\ldots,X_n)= f^{\wt}(P_{m_1}, \dots, P_{m_d}).$		 
	Observe that  $f^{\wt}= b_iY_i$ where $b_i:=\max\{b_k,\,  b_k\neq 0, \, 1\leq k \leq n\}$. Following the proof of Proposition \ref{prop: singular locus f linear} we deduce that $ \Sigma_g^{\infty} \subset \A^n$ has dimension at most $0$. Thus, $\mathrm{pcl}(V_g)\subset \Pp^n$ has no singular points at infinity.
	 
\end{proof}

From Proposition \ref{prop: singular locus f linear} and  Corollary \ref{singular locus plc(Vg) at infinity with f linear} we conclude the following result.

\begin{theorem} \label{singular locus plc(Vg) with f linear}
Let $n \geq 3$. Let $f \in \fq[Y_1, \dots, Y_d]$ be a linear polynomial and $g \in \fq[X_1, \dots, X_n]$ defined as in \eqref{def: g} or $g\in \fq$. Then  $\mathrm{pcl}(V_g)\subset \Pp^n$ has singular locus of dimension at most $0$.
\end{theorem}
From Theorem \ref{singular locus plc(Vg) with f linear} and following the proof of Corollary \ref{teo: V es absolutamente irreducible}, we obtain:

\begin{corollary}\label{teo: V es absolutamente irreducible f linear}
	The hypersurface $V_g$ is absolutely irreducible.
\end{corollary}
We arrive at the following result  concerning the number of  $\fq$-rational points of the variety $V_g$ for a linear polynomial $f$. 
\begin{theorem}  \label{estimation with f linear} 
Let $n\geq 3$. Let $m_{1},\ldots, m_{d}$ be positive integer with  $ m_{1}<\cdots<  m_{d}$. We assume that $\mathrm{char}(\fq)$ does not divide $e$ and $m_{j}$  for all $1\leq j \leq d$. Let $R_g=f(P_{m_1}, \dots, P_{m_d})+g$ with  $f=b_1 Y_1+ \dots+b_dY_d+a$ and $g \in \fq[X_1 \klk X_n]$ of degree $e$ and $g$ is defined as in \eqref{def: g}  or $g\in \fq$. Then,  
	\begin{equation*} \label{estimation case lineal }
	\big||V_g(\fq)|-q^{n-1}\big|\leq q^{(n-1)/2}\big(2(m_i-1)^{n-1}q^{1/2}+6(m_i+2)^n\big),
	\end{equation*}
	where  $m_i:=\max\{m_k,\,  b_k\neq 0, \, 1\leq k \leq n\}.$
\end{theorem}
\begin{proof}
	From Theorem  \ref{singular locus plc(Vg) with f linear}, Corollary \ref{teo: V es absolutamente irreducible f linear} and  the estimate \eqref{eq: estimacion de Ghorpade Lachaud},  we have 
	\begin{equation}\label{estimation f linear}
	\big||\mathrm{pcl}(V_g)(\fq))|-p_{n-1}\big | \leq (m_i-1)^{n-1} q^{n/2}+6(m_i+2)^nq^{(n-1)/2}.
	\end{equation}
Now, we estimate the number of $\fq$-rational points of $V_g^{\infty}:=\mathrm{pcl}(V_g)\cap\{X_0=0\}\subset \Pp^{n-1}$. Note that $V_g^{\infty}=V(R_g^{\deg R_g})$ is a hypersurface of $\Pp^{n-1}$ of dimension $n-2$. Following  the proof of Corollary \ref{singular locus plc(Vg) at infinity with f linear} we deduce that $V_g^{\infty}$ is nonsingular variety of degree $\deg(R_g)=m_i$. 
	Then,
from \eqref{estimacion Deligne}:
	\begin{equation}\label{estimation V infty f linear}
	||V_g^{\infty}(\fq)|-p_{n-2}| \leq (m_i-1)^{n-1} q^{(n-2)/2}.
	\end{equation}
	From \eqref{estimation f linear} and \eqref{estimation V infty f linear}, we conclude that
	\begin{align*}
	\big||V_g(\fq)|-q^{n-1}\big|\leq & \big||\mathrm{pcl}(V_g)(\fq)|-p_{n-1}\big|+\big||V_g^{\infty}(\fq)|-p_{n-2}\big|\\
	\leq & (m_i-1)^{n-1} q^{n/2}+6(m_i+2)^nq^{(n-1)/2}+(m_i-1)^{n-1} q^{(n-2)/2}\\
	\leq & q^{(n-1)/2}\big(2(m_i-1)^{n-1}q^{1/2}+6(m_i+2)^n\big).
	\end{align*}	
\end{proof}

\begin{remark}
Note that Theorem \ref{estimation with f linear} improves Theorem \ref{teo 1 intro} when $f$ is a nonzero linear polynomial. Indeed, the estimate of Theorem \ref{estimation with f linear} does not depend on $d$, the number of $m_j$-power sum polynomials in which $f$ is evaluated. Concretely, from Theorem \ref{teo 1 intro} we have that $|V_g(\fq)|=q^{n-1}+ \mathcal{O}(q^{(n+d)/2})$, while  Theorem \ref{estimation with f linear} implies  $|V_g(\fq)|=q^{n-1}+ \mathcal{O}(q^{n/2})$. 
\end{remark}

\begin{remark}
Suppose that the polynomial $g_1$ of the definition \eqref{def: g} is identically zero 
and $R_g$ is an homogeneous polynomial. Then $R_g=c(X_1^e+\cdots+X_n^e)$ for some $c\in \fq$ and $V_g \subset \Pp^{n-1}$. It is easy to see that the set of $\{{\bf{x}}\in \A^n:\,\,\nabla R_g({\bf{x}})=0\}=\{{\bf{0}}\}$, from where $V_g$  is a nonsingular variety. 
Thus, from \eqref{estimacion Deligne}:
\begin{equation*} 
\big|\overline{N}_g-p_{n-2}\big|\leq (e-1)^{n-1}q^{(n-2)/2},
\end{equation*}
where  $\overline{N}_g$ is the number of $\fq$--rational projective points of $V_g.$
%
As in Remark \ref{remark: g1=0}, we conclude that
\begin{equation*} 
\big||V_g(\fq)|-q^{n-1}\big|\leq (e-1)^{n-1}q^{(n-2)/2}(q-1).
\end{equation*}
%
\end{remark}
\begin{remark} Suppose that $g=0$ and  $R_0$ is an homogeneous polynomial. Then  $R_0=c(X_1^m+\cdots+X_n^m)$ with $c\in \fq$. Following the arguments of the  above remark, it can be shown that 
\begin{equation*} 
\big|N_0-q^{n-1}\big|\leq (m-1)^{n-1}q^{(n-2)/2}(q-1).
\end{equation*}
\end{remark}
\section{Special Deformed Hypersurfaces}

In this section we follow the same methodology to obtain estimates of $\fq$--solutions of some well known equations over finite fields. These results are not obtained directly from applying  Theorem \ref{teo 1 intro} since we can take advantage of the properties of the polynomial $f$ under consideration for these particular cases. 
\subsection{ Deformed diagonal equations over a  finite field}
Let $m$, $n$ be positive integers such that $n\geq 3$ and $m\geq 2$ is not divisible by $\mathrm{char}(\fq)$  and let $g \in \fq[X_1 \klk X_n]$ with $\deg(g)< m$. Consider the equation:
\begin{equation}\label{eq: deformed diagonal}
c_1X_1^m + \dots+ c_nX_n^m +g(X_1 \klk X_n)=0,
\end{equation}
where $c_i\in \fq \setminus \{0\}$, $1\leq i \leq n$.
We denote by $N_g$ the number of $\fq$--rational solutions of \eqref{eq: deformed diagonal}. Let  $R_g:=c_1X_1^m + \dots+ c_n X_n^m +g(X_1 \klk X_n)$ and let $V_g\subset \A^n$ be defined by $V_g=V(R_g)$. In this case, ${\bf{x}}\in \Sigma_g$ satisfies
\begin{equation*} R_g=0, \quad c_jmX_j^{m-1}+\frac{\partial g}{\partial X_j}=0,\quad 1\leq j \leq n.
\end{equation*}
Following the same arguments used to prove Proposition \ref{prop: singular locus f linear}, Corollaries \ref{singular locus plc(Vg) at infinity with f linear} and \ref{teo: V es absolutamente irreducible f linear} and Theorems \ref{singular locus plc(Vg) with f linear} 
  and \ref{estimation with f linear}   we can deduce  the following result.
\begin{theorem}\label{theo: estimacion ecuaciones diagonales deformadas}
With the hypotheses as above, we have that
	\begin{equation} \label{estimation case d=1 }
	|N_g-q^{n-1}|\leq q^{(n-1)/2}\big( 2(m-1)^{n-1}q^{1/2}+6(m+2)^n \big).
	\end{equation}
\end{theorem}

Observe that in \cite{AdSp87} the authors use  the Newton polyhedra to prove a result that allows one to obtain an estimate on  the number of $\fq$-solution of a deformed diagonal equation. This result holds under some hypotheses for  $g$, which  are not present in Theorem  \ref{theo: estimacion ecuaciones diagonales deformadas}.


\begin{theorem}\label{teo:existence}
	Let $q>(m+2)^{\frac{2n}{n-2}}$ and $g\in \fq[X_1,\ldots,X_n]$ of degree less than $m$. Assume that $\mathrm{char}(\fq)$ does not divide $m$. Then, the equation $c_1X_1^m+\cdots+c_nX_n^m+g(X_1,\ldots,X_n)=0$, has at least one solution in  $\fq^n$.
\end{theorem}
\begin{proof}
Observe that if $q > 144$ then $6(m+2)^n < 1/2 (m+2)^n q^{1/2}$.   On the other hand, we have that $2 (m-1)^{n-1} < 1/2 (m+2)^n$. Then, from \eqref{estimation case d=1 } we deduce that 
$$|N_g-q^{n-1}|\leq q^{n/2}(m+2)^n,$$
	from where 
	\begin{equation}\label{Ng}
	N_g \geq q^{n-1}-q^{n/2}(m+2)^n.
	\end{equation}	
	Therefore, the equation $R_g=c_1X_1^m+ \cdots +c_nX_n^m+g(X_1,\cdots,X_n)=0$ has at least one solution in $\fq^n$  if the
	right-hand side of \eqref{Ng} is a positive number. The result follows. 
\end{proof}
\begin{remark}
Note that if $q>(m+2)^2$ then $n> \frac{2 \log(q)}{\log(\frac{q}{(m+2)^2})}$. Observe that $E(q)=\frac{2\log(q)}{\log(\frac{q}{(m+2)^2})}$ is a decreasing function and $\lim_{q\rightarrow \infty}E(q)=2$. Hence, for $q$ sufficiently large, the equation \eqref{eq: deformed diagonal} has at least one solution in $\fq^n$, $n>2$.
\end{remark}
Carlitz \cite{Ca56} showed that if $m=n$, $m$ divides $\mathrm{char}(\fq)-1$ and $g\in\fq[X_1,\ldots,X_n]$ is of degree less than $m$, then $R_g=0$ has at least one solution in $\fq^n$. This result was extended by  Felszeghy \cite{Fe06}, who proved that $R_g=0$ has at least one solution in $\fq^n$ if $q=p:=\mathrm{char}(\fq)$ and 
	\begin{equation}\label{existence Felszeghy}
	 n \geq \Big\lceil \frac{p-1}{\lfloor \frac{p-1}{m}\rfloor}\Big \rceil. 
\end{equation}	
He also shows that if $m$ divides $p-1$ and $n\geq m$, then $R_g=0$ is solvable in $\fq^n$.	

	Theorem \ref{teo:existence} improves \eqref{existence Felszeghy} in several aspects. Indeed, on one hand, our result holds for any $q$ such that $\mathrm{char}(\fq)$  does not divide $m$ while \eqref{existence Felszeghy} holds if $q=p$.  On the other hand, we prove that if $q>(m+2)^2$ then \eqref{eq: deformed diagonal} has at least a $\fq$-rational solution for $n\geq 3$ while Felszeghy's result requires $p\geq m+1$ and $n\geq m+1$. In particular, if $p>(m+2)^2$ we can guarantee the existence of at least one $\fp$-rational solution for any $n\geq 3$ instead of $n\geq m+1$.


\begin{remark} \label{equation diagonal} For $b\in \fq$ consider the following diagonal equation:
$$c_1X_1^m+\cdots+ c_nX_n^m=b.$$
If $b=0$ then Theorem \ref{estimation with f linear} gives us $N_0=q^{n-1}+\mathcal{O}(q^{n/2})$, which is of similar order to the results in the literature (see, e.g.,  \cite[Chapter 6, \S 3]{LiNi83}). Suppose then $b\neq 0$. From \eqref{estimation case d=1 }, we have that $N_b=q^{n-1}+\mathcal{O}(q^{n/2}),$ but this estimate can be improved. Indeed, it can be shown that the singular locus $\Sigma_b$ of the affine variety $V(R_b)$ defined by $R_b=c_1X_1^m+\cdots+ c_nX_n^m-b$ is $\Sigma_b=\{0\}.$ Hence, the projective variety $\mathrm{pcl}(V(R_b))$ is nonsingular and, from \eqref{estimacion Deligne},  we have 
\begin{equation}\label{eq diagonal b distinto 0}
|N_b-q^{n-1}| \leq (m-1)^nq^{(n-2)/2}(1+q^{1/2}).
\end{equation}
Observe that this result is  Weil 's estimate for diagonal equations (see, e.g., \cite[Chapter 6, \S 3]{LiNi83}).
\end{remark} 

\subsection{Generalized Markoff-Hurwitz-type equations}
Let $m, n, k_1, \ldots, k_n$ be positive integers, $n\geq 3$ and $m\geq 2$, $a, b \in \fq$ and $a_i \in \fq\setminus \{0\}$, $1\leq i \leq n$. Consider the equation
\begin{equation}\label{Baoulina}
a_1X_1^m+\cdots +a_nX_n^m+a=bX_1^{k_1}\ldots X_n^{k_n}.
\end{equation}
 Observe that this is an special case of deformed diagonal equations. Denote by $N$ the number of $\fq$--rational solutions of \eqref{Baoulina}. Suppose that $\mathrm{char}(\fq)$ does not divide $m$. 
 Let  $R_g:=a_1X_1^m+\cdots +a_nX_n^m+g$, where $g=a-bX_1^{k_1}\cdots X_n^{k_n}$ and $k_1+\cdots +k_n<m$. From Theorem \ref{theo: estimacion ecuaciones diagonales deformadas}
we obtain the following result.
\begin{theorem} \label{teo: estimacion eq de Baoulina}
  With the same hypotheses as above, $N$ satisfies the following estimate:  
\begin{equation*} 
	\big|N-q^{n-1}\big|\leq q^{(n-1)/2}\big(2(m-1)^{n-1}q^{1/2}+6(m+2)^n \big).
	\end{equation*}
\end{theorem}
In what follows we obtain sufficient conditions for the existence of a $\fq$--rational solution with nonzero coordinates. 
We shall need the following estimate on the number of $\fq$--rational solutions of \eqref{Baoulina} with $i$ coordinates equal to zero. We denote this number $N_i$.

	\begin{proposition}  With the same hypotheses as above, the number $N_i$ satisfies the following estimate:

If $a=0$  and $i=1 \klk n-2$, then
\begin{equation} \label{cota sup si a es cero}
|N_i-q^{n-i-1}| \leq (m-1)^{n-i-1}q^{(n-i-2)/2}(q-1).
\end{equation}

If $a\neq 0$ and $i=1 \klk n-1$, then 
\begin{equation}\label{cota sup  para a distinto de cero}
|N_i-q^{n-i-1}| \leq (m-1)^{n-i}q^{(n-i-2)/2}(1+q^{1/2})
\end{equation}
	\end{proposition}
	\begin{proof}
	Follows from  \eqref{Baoulina} and Remark \ref{equation diagonal}.	
%
	\end{proof}
Let $N^*$ be the number of $\fq$-rational solutions of \eqref{Baoulina} with nonzero coordinates and let $N^{=}$ be the number of $\fq$-rational solutions of \eqref{Baoulina} with at least one coordinate equal to zero. Note that $N^* =N- N^{=}$. 
By the inclusion-exclusion principle we obtain
\begin{equation}\label{Pcio inclusion exclusion}
N^{=}=\sum_{i=1}^n (-1)^{i+1} \binom{n}{i} N_i.
\end{equation}


Suppose that $a \neq 0$. 
From \eqref{Pcio inclusion exclusion}  and since $N_n=0$, we have
\begin{align*}
N^*-\frac{(q-1)^n}{q}&=N-N^{=}-\sum_{i=0}^n (-1)^i\binom{n}{i}q^{n-i-1}\\
&=N-\sum_{i=1}^n (-1)^{i+1} \binom{n}{i} N_i-\sum_{i=0}^n (-1)^i\binom{n}{i}q^{n-i-1}\\
&=N-q^{n-1}+\sum_{i=1}^{n-1}(-1)^i \binom{n}{i} (N_i-q^{n-i-1})+(-1)^n(N_n-q^{-1})\\
&=(N-q^{n-1})+\sum_{i=1}^{n-1}(-1)^i \binom{n}{i} (N_i-q^{n-i-1})-(-1)^n \frac{1}{q}.
\end{align*}
Thus,  we deduce that
\begin{equation*}
N^*-\frac{(q-1)^n-(-1)^n}{q}=(N-q^{n-1})+\sum_{i=1}^{n-1}(-1)^i \binom{n}{i} (N_i-q^{n-i-1}).
\end{equation*}
Therefore, from Theorem \ref{teo: estimacion eq de Baoulina} and \eqref{cota sup  para a distinto de cero}:
\begin{align*}\label{a distinto cero}
\Big|N^*-\frac{(q-1)^n-(-1)^n}{q}\Big| & \leq |N-q^{n-1}|+\sum_{i=1}^{n-1}\binom{n}{i}|N_i-q^{n-i-1}|\\
& \leq 6(2m)^n q^{n/2}+2 m^{n-1} \sum _{i=1}^{n-1}\binom{n}{i} q^{(n-i-1)/2}\\
&\leq 6(2m)^n q^{n/2}+ 2 m^{n-1} \frac{(\sqrt{q}+1)^n-\sqrt{q}^n-1}{\sqrt{q}} \\
&\leq 6(2m)^n q^{n/2}+ \frac{2}{m}(2 m)^n q^{\frac{n-1}{2}}\\
& \leq 6(2m)^nq^{n/2}+ (2m)^n q^{n/2}\\
& \leq 7(2m)^n q^{n/2}.
\end{align*}
We have proved the following result.
\begin{proposition}\label{estimation a distinto de cero}
 With the hypotheses as above and $q>2$, the number $N^*$ of $\fq$--rational solutions of \eqref{Baoulina} with nonzero coordinates  satisfies the following estimate:
 $$\Big|N^*-\frac{(q-1)^n-(-1)^n}{q}\Big|   \leq 7(2m)^n q^{n/2}.$$
\end{proposition}

In \cite{Mo63}   Mordell studies the following equation:
\begin{equation*}\label{Mordel}
(a_1X_1^{m_1}+\cdots+ a_nX_n^{m_n}+a)^k=bX_1^{k_1}\cdots X_n^{k_n},
\end{equation*}
 where $a_1,\ldots, a_n,a$ are nonzero elements of $\fq$, $k,m_1,\ldots,m_n\in \N$, $k_1,\ldots,k_n$ are nonnegative integers and $n\geq 2$. The author shows that if $k_1=\cdots = k_n=1$ and $q=p$ then
\begin{equation}\label{Mordell}\Big|N^*-\frac{(q-1)^n}{q}\Big| \leq d_1 \cdots d_n q^{n/2},\end{equation} where $d_i=\gcd(m_i,q-1).$ 
Observe that Proposition \ref{estimation a distinto de cero}  improves \eqref{Mordell} when $k=1$ and $m=m_1=\cdots = m_n$. Indeed, our result holds for all $q>2$ with $\mathrm{char}(\fq)$  not dividing $m$ and  $m>k_1+\cdots+k_n$. 
Moreover,  we determine one more term in the asymptotic development in terms of $q$. Namely, we prove $N^*= \frac{(q-1)^n-(-1)^n}{q}+ \mathcal{O}(q^{n/2})$ instead of
$N^*=\frac{(q-1)^n}{q} + \mathcal{O}(q^{n/2}).$ Observe that this term appears in the asymptotic development of $N^*$ when $a=0$ (see, e.g., \cite[Theorem 3.1]{Ba15}).
Now, we provide an existence result for $\fq$--rational solutions with nonzero coordinates.
From Proposition \ref{estimation a distinto de cero} we deduce that
\begin{align*}
N^* &\geq \frac{(q-1)^n-(-1)^n}{q}-7(2m)^n q^{n/2}\\
& \geq \frac{(q-1)^n}{2q}-7(2m)^n q^{n/2}\\
& \geq  q^{(n-2)/2}(2m)^n \Big( \frac{(q-1)^{n}}{2(2m)^nq^{n/2}}-7q\Big) \\
& \geq q^{(n-2)/2}(2m)^n \Big( \frac{1}{2 (2m)^n}\big(\frac{(q-1)^2}{q}\big)^{n/2}-7q\Big) \\
& \geq q^{(n-2)/2}(2m)^n \Big( \frac{(q-2)^{n/2}}{2(4m^2)^{n/2}}-7q\Big) 
\end{align*}
From  Bernoulli's inequality 
we obtain
\begin{align*}
N^* & \geq  q^{(n-2)/2}(2m)^{n} \Big( \frac{n}{4}\Big(\frac{q-2-4m^2}{4m^2}\Big)-7q\Big) 
\end{align*}
Therefore, \eqref{Baoulina} has at least one solution in $\fq^n$ with nonzero coordinates if 
$$\frac{n}{4}\Big(\frac{q-2-4m^2}{4m^2}\Big)-7q>0.$$
That is,
$$n(q-2-4m^2) -112m^2q >0,$$
which is equivalent to
$$ (q-2-4m^2)(n-112m^2)- 112m^2(2+4m^2)>0.$$
We conclude that if $n > 112m^2$  the equation  \eqref{Baoulina} has at least a nontrivial solution in $(\fq^*)^n$  for
$$q >\frac{112 m^2 (2+4m^2)}{n-112m^2}+2+4m^2.$$
We have proved:
\begin{proposition}\label{condition existence eq Baoulina}
If   $q >\frac{112 m^2 (2+4m^2)}{n-112m^2}+2+4m^2$, $\mathrm{char}(\fq)$ does not divide $m$ and $n > 112m^2$ then the equation \eqref{Baoulina} has at least one solution in $(\fq^*)^n$ for the case $a\neq 0$.
\end{proposition}
\begin{remark}
It can be shown that if  $q> (2m)^{\frac{2n}{n-4}}$, $m \geq 3$ and $n \geq 5$, then the equation \eqref{Baoulina} has at least one solution in $(\fq^*)^n$. This existence result holds for far more values of $n$ although it requires larger values of $q$.
\end{remark}

\begin{remark}\label{estimation a equals 0}
 Suppose that $a = 0$. With  the same hypotheses and arguments of  Theorem \ref{teo: estimacion eq de Baoulina} and taking into account \eqref{cota sup si a es cero}, we deduce that the number $N^*$ of $\fq$--rational solutions of \eqref{Baoulina} with nonzero coordinates  satisfies the following estimate:
\begin{equation} \label{estimacion a=0 coord no nulas}
\Big|N^*-\frac{(q-1)^n}{q}-(-1)^n(n+1-\frac{1}{q})\Big| \leq 7(2m)^n q^{n/2}.
\end{equation}
\end{remark}

For $a=0$, $k_1=\cdots=k_n=1$, $n \geq 2$ and $1 \leq m\leq q-1$, Baoulina \cite{Ba15} shows that
\begin{equation}\label{estimation Baoulina}\Big|N^*-\frac{(q-1)^n-(-1)^n}{q}\Big| \leq (d_0 d_1^{n-1}-1)q^{(n-1)/2}+(d-d_0)d_1^{n-1}q^{(n-2)/2},
\end{equation}
where $d_1=\gcd(m,q-1)$, $d=\gcd(n-km,\frac{q-1}{d_1})$ and $d_0=\gcd(d,k)$.
Although \eqref{estimation Baoulina} is better than \eqref{estimacion a=0 coord no nulas} when $k_1=\cdots=k_n=1$,  our estimate extends   \eqref{estimation Baoulina} for the case $k_1+\cdots+k_n<m$. Moreover, we determine one more term in  the asymptotic development of $N^*$ in terms of $q$, namely  $N^*= \frac{(q-1)^n-(-1)^n}{q}+(-1)^n(n+1)+ \mathcal{O}(q^{n/2})$.

\subsection{Carlitz's equations}
 Let  $d,n$ be positive integers with $d\geq 2$ and $n\geq 3$. Let $h_i=a_{d,i} T^d+\cdots +a_{0,i} \in \fq[T]$, with $\deg(h_i)=d$, $1\leq i\leq n$. Let $g\in \fq[X_1,\ldots,X_n]$ such that $\deg (g)< d$. Suppose that $\mathrm{char}(\fq)$  does not divide $d$. Carlitz's equations are:
\begin{equation}\label{eq. carlitz}
h_1(X_1)+\cdots +h_n(X_n)=g.
\end{equation} 
Denote by $N$ the number of $\fq$-rational solutions of \eqref{eq. carlitz}. Let $R_g, V_g$ and $\Sigma_g$ be defined as usual. 
 A given ${\bf{x}}\in \Sigma_g$ satisfies the following equations: 
$$Q_j:=h_j'(X_j)-\frac{\partial g}{\partial X_j}=0,\,\,\,1\leq j \leq n.$$
By the same arguments as before we can show that $Q_j$, $1\leq j \leq n$, form a regular sequence of $\fq[X_1,\ldots,X_n]$ and $\mathrm{pcl}(V_g)$ has no singular points at infinity. Thus, we obtain the following result.
\begin{theorem}
Let $n\geq 3$ and  $d\geq 2$ not divisible by $\mathrm{char}(\fq)$. Then the singular locus of $\mathrm{pcl}(V_g)$ has dimension at most $0$ and $V_g$ is absolutely irreducible.
\end{theorem}
Following the same reasoning as in the proof of Theorem  \ref{estimation with f linear} we conclude: 
\begin{theorem}\label{estimation N}
	Let $n\geq 3$ and  $d\geq 2$ not divisible by $\mathrm{char}(\fq)$. Then 	
	\begin{equation}\label{estimation Nbis}
	\big|N-q^{n-1}\big|\leq   q^{(n-1)/2}\big(2(d-1)^{n-1}q^{1/2}+6(d+2)^n\big).
	\end{equation}
\end{theorem} 
\begin{remark} Observe that for $h_1=\cdots=h_n=h$ then the equation \eqref{eq. carlitz} can be written as  $$a_1 P_1+ \cdots +a_dP_d+ na_0=g.$$ Thus, \eqref{eq. carlitz} can be expressed as $f(P_1,\dots, P_d)+ na_0=g$, where $f\in \fq[Y_1, \dots, Y_d]$ is the linear polynomial $f:=a_1Y_1+\dots+a_dY_d +na_0$. We can then  apply Theorem \ref{estimation with f linear} to obtain a similar result as the one in the previous theorem. 
\end{remark}

Carlitz \cite{Ca53} studies the number of $\fq$--rational solutions of the equation $$h_1(X_1) +\cdots+h_n(X_n)=\alpha,$$ where $h_i \in \fq[T]$ with $2 <\deg(h_i)=k_i<\mathrm{char}(\fq)$ and $\alpha \in \fq$. More precisely, he  proves that the number $N$ of $\fq$-rational solutions of this equation is given by
\begin{equation}\label{estimation Nbisbis}
N=q^{n-1}+ \mathcal{O}(q^{n-w}),
\end{equation} 
where $w=\frac{1}{k_1}+ \cdots+\frac{1}{k_n}$ and the constant implied by the $\mathcal{O}$ is not explicitly given.
Theorem \ref{estimation N} improves \eqref{estimation Nbisbis} in several aspects if $\deg(h_i)=d$, $1\leq i\leq n$. On one hand,  Theorem \ref{estimation N} gives an explicit estimate on the number $N$. On the other hand, the equation can be matched to a non-necessarily constant polynomial.
Finally, our result implies $N=q^{n-1}+\mathcal{O}(q^{n/2})$ while  \eqref{estimation Nbisbis} implies that $N=q^{n-1}+ \mathcal{O}(q^{n-w})=q^{n-1}+\mathcal{O}(q^{n/2+\epsilon})$,  where $\epsilon=n(\frac{1}{2}-\frac{1}{d})>0$ and $w=\frac{n}{d}.$

Regarding existence results, Carlitz \cite{Ca56} shows that if  $\deg(h_i)=n$, $1\leq i \leq n$ and $n$ divides $\mathrm{char}(\fq)-1$, then \eqref{eq. carlitz} has at least one solution in $\fq^n$. From Theorem \ref{estimation N}  we can find conditions which imply that \eqref{eq. carlitz}  has at least one solution in $\fq^n$ when $d\geq 2$ and $n\geq 3$.
\begin{theorem}
	Let $q > (d+2)^{\frac{2n}{(n-2)}}$, $ d \geq 2$ not divisible by $\mathrm{char}(\fq)$  and $ n\geq 3$. Let  $h_i \in \fq[T]$ be defined by $h_i=a_{d,i} T^d+\cdots +a_{0,i}$, $1\leq i \leq n$, with $\deg(h_i)=d$ and $g\in \fq[X_1,\ldots,X_n]$ with $\deg(g)<d$. Then the equation $h_1(X_1)+ \cdots + h_n(X_n)=g$ has at least one solution in $\fq^n$.
\end{theorem}
\begin{remark} Theorem \ref{estimation N} provides an upper bound of Waring's number for univariate polynomials over $\fq$.
 Waring's problem consists in finding the minimum number of variables such that the equation $X_1^d + \dots+X_n^d=\beta$ has solutions for any natural number $\beta$. This minimum number is called the \textit{Waring's number associated to $d$}. The Waring's problem has also been considered for equations over finite fields and there are many bounds for their Waring number (see, e. g.,  \cite[Chapter 13]{MuPa13}). 

Consider the following generalization of Waring's problem: given a polynomial $h \in \fq[T]$ of degree $d$, find the minimum number of variables such that
\begin{equation}\label{waring problem 1}
h(X_1)+ \dots+h(X_n)=\beta
\end{equation}
has a solution in $\fq^n$ for any $\beta \in \fq$.  We denote this number by $\gamma(h,q)$. Carlitz  \cite{CaLeMiSt61} proves that if $q$ is a prime number then $\gamma(h,q) \leq d$  whenever $d \neq p-1, \frac{p-1}{2}$. On the other hand, Castro et. al.  \cite{CaRuVe08}  obtain an upper bound on $\gamma(h,q)$ for polynomials of the form $h=aT^d+g \in \fq[T]$, where $d \neq 1$ divides $p-1$ and  $g$ satisfies  certain hypothesis  related to $p$-weight degree of $g$.

Let $N$ be  the  number of $\fq$--rational solutions of equation \eqref{waring problem 1}. Suppose that $d \geq 2$ and $n \geq 3$. From Theorem \ref{estimation N} we have that
$$\big|N-q^{n-1}\big|\leq   q^{(n-1)/2}\big(2(d-1)^{n-1}q^{1/2}+6(d+2)^n\big).$$

On one hand, $q > 144$ implies that $6(d+2)^n < 1/2 (d+2)^n q^{1/2}$. On the other hand, $2 (d-1)^{n-1} < 1/2 (d+2)^n$ holds. Thus, $N>0$ provided that $q^{n/2}\big(q^{(n-2)/2}-(d+2)^n)>0;$ that is
 $q^{(n-2)/2} >(d+2)^n$. 
Now, if $q > (d+2)^2$ the condition $q^{(n-2)/2} >(d+2)^n$  is equivalent to
$$n > \frac{log(q^2)}{log(\frac{q}{(d+2)^2})}.$$
We conclude that  if $n > \frac{log(q^2)}{log(\frac{q}{(d+2)^2})}$ then 
$\gamma(h,q) \leq \Big\lceil  \frac{log(q^2)}{log(\frac{q}{(d+2)^2})} \Big\rceil.$ On the other hand, if $q > (d+2)^{2(d-1)/(d-3)}$, then we obtain that $\Big\lceil  \frac{log(q^2)}{log(\frac{q}{(d+2)^2})} \Big\rceil \leq d.$ Thus, we have the following result.
\begin{theorem}
	Let $n\geq 3$ and $d\geq 4$. Assume that $\mathrm{char}(\fq)$ does not divide $d$. For any $q > (d+2)^{2(d-1)/(d-3)}$ and any $h \in \fq[T]$ of degree $d$, we have that
	$$\gamma(h,q) \leq \Big\lceil  \frac{log(q^2)}{log(\frac{q}{(d+2)^2})} \Big\rceil$$
\end{theorem}
Note in particular that solutions with small number of variables require large values of $q$.
\end{remark}

\subsection{Equations in Dickson polynomials}
Let $d \in \mathbb{N}$ and $a \in \fq$. The Dickson polynomials over $\fq$ of degree $d$ with parameter $a$ are:
$$D_d(X,a)=\sum_{i=0}^ {\lfloor d/2 \rfloor} \frac{d}{d-i} \binom{d-i}{i} (-a)^i X^{d-2i} .$$
Dickson polynomials have been extensively studied  because they play very important roles in both theoretical work as well as in various applications (see, \cite[Chapter 7]{MuPa13}). 

Let $d, n$ be positive integers with $d \geq 2$ and $n \geq 3$ and  assume that $\mathrm{char}(\fq)$ does not divide $d$. Let $g \in \fq[X_1, \dots, X_n]$ be such that $\deg(g) < d$. Consider the following  polynomial equation
\begin{equation}\label{eq: dickson}
c_1 D_d(X_1,a_1)+ \dots+c_n  D_d(X_n,a_n)=g,
\end{equation}
where $a_1 \klk a_n \in \fq$ and $c_1 \klk c_n \in \fq \setminus \{0\}$. Observe that for $a=0$ this is a deformed  diagonal equation.
Equation \eqref{eq: dickson} is a particular case of Carlitz's equations defined in \eqref{eq. carlitz} for $h_i:=D_d(X_i,a_i).$ From Theorem \ref{estimation N} we have:
\begin{theorem}\label{estimation N disckson}
Let $n\geq 3$ and $d\geq 2$ not divisible by $\mathrm{char}(\fq)$. Then 
\begin{equation}\label{estimation Nbis dickson}
\big|N-q^{n-1}\big|\leq   q^{(n-1)/2}\big(2(d-1)^{n-1}q^{1/2}+6(d+2)^n\big),
\end{equation}
where $N$ denotes the number of $\fq$--rational solutions of \eqref{eq: dickson}.
\end{theorem}

\begin{remark}

Observe that if $a_1= \dots=a_n=a$ and $c_1= \dots=c_n=c$, then the equation \eqref{eq: dickson} can be written as follows
$$\sum_{i=0}^{\lfloor d/2 \rfloor}b_iP_{d-2i}=g,$$
where $b_i= \frac{d}{d-i} \binom{d-i}{d} (-a)^i$ and $P_{d-2j}=X_1^{d-2j} + \dots +X_n^{d-2j}$ for $0\leq j \leq \lfloor d/2 \rfloor$.
Thus, equation \eqref{eq: dickson} can be expressed as $$f(P_{d-2\lfloor d/2 \rfloor} \klk P_d)=g,$$ where $f \in \fq[Y_{d-2\lfloor d/2 \rfloor}, \dots, Y_d]$ is the linear polynomial $f:=\sum_{i=0}^{\lfloor d/2 \rfloor}b_iY_{d-2i}$. Hence, applying  Theorem 
\ref{estimation with f linear} we obtain a similar result as the one in the previous theorem.
\end{remark}

In \cite{ChMuWa08} the authors study the number  of $\fq$--rational solutions of the equation
$$c_1D_{d_1}(X_1,a_1)+ \dots +c_nD_{d_n}(X_n,a_n)= c, $$
where $d_1, \dots, d_n $ are positive integers, $c_1, \dots,c_n \in \fq \setminus \{0\}$ and $c,a_1, \dots,a_n \in \fq$. More precisely, they prove that if $ n, d_1 \klk, d_n \geq 2$  and there exists $ 0 \leq t \leq n$ such that $a_1=\dots =a_t=0 $ and $a_j \neq 0$ for all $t <j \leq n$, then the number $N$  of $\fq$-rational  solutions of these equations  satisfies:
\begin{equation}\label{Estimation Chow}|N-q^{n-1}| \leq q^{(n-2)/2}(q-1) \prod_{j=1}^t(m_j-1) \prod_{j=t+1}^n (m_j+l_j),\end{equation}
where $m_j=\gcd(d_j,q-1)$ and $l_j=\gcd(d_j,q+1)$ for $1 \leq j\leq n$.
If  $d_1=\dots=d_n=d$ then Theorem \ref{estimation N disckson} extends the estimation  in \eqref{Estimation Chow} in the sense that it holds for $c \in \fq[X_1, \dots, X_n]$ with $0 <\deg(c)< d$.


%
%
%
%
%
%
%
%

Finally, from  Theorem \ref{estimation N disckson} we derive conditions about the solvability of \eqref{eq: dickson} in $\fq^n$.
\begin{theorem}  \label{ex dickson}
	Let $q >\frac{4}{9} (d+2)^{\frac{2n}{(n-2)}} $, $d\geq 2$ not divisible by $\mathrm{char}(\fq)$  and $n\geq 3$. Let $g\in [X_1,\ldots,X_n]$, $\deg(g)<d$. Then the equation $$c_1D_d(X_1,a_1)+ \cdots + c_nD_d(X_n,a_n)=g$$ has at least one solution in $\fq^n$.
\end{theorem}
\begin{proof}
	From \eqref{estimation Nbis dickson}, we have that
	$$\big|N-q^{n-1}\big|\leq   q^{(n-1)/2}\big(2(d-1)^{n-1}q^{1/2}+6(d+2)^n\big).$$	
Observe that if $q > 36^{2}$ then $6(d+2)^n< 1/6 (d+2)^n q^{1/2}$.   On the other hand, $2 (d-1)^{n-1} < 1/2 (d+2)^n$. Under this condition we have that
	$$\big|N-q^{n-1}\big|\leq  \frac{2}{3} q^{n/2}(d+2)^n.$$ 	
	Thus, $N$ satisfies the following inequality:
	\begin{equation}\label{upper bound N dicskon}
	N \geq q^{n-1}- \frac{2}{3}q^{n/2}(d+2)^n= q^{n/2}( q^{(n-2)/2}-\frac{2}{3}(d+2)^n).
	\end{equation}
We conclude that \eqref{eq: dickson} has at least one solution in $\fq^n$ if  $q^{(n-2)/2}> \frac{2}{3}(d+2)^n$.

\end{proof}

\begin{remark}

Suppose that $d_1=\dots=d_n=d$. Observe that, when $g$ is a constant, Theorem \ref{ex dickson} gives similar conditions on $q$, $d$ and $n$ than \cite[Theorem 11]{ChMuWa08}, under which there exists at least one $\fq$--rational solution of \eqref{eq: dickson}. Theorem \ref{ex dickson} extends \cite{ChMuWa08} since it holds for  a polynomial  $g$ of positive degree at most $d$.
\end{remark}

\textbf{Acknowledgements.}  The authors are pleased to thank Guillermo Matera for several valuable comments and suggestions.

\newcommand{\etalchar}[1]{$^{#1}$}
\providecommand{\bysame}{\leavevmode\hbox
	to3em{\hrulefill}\thinspace}
\providecommand{\MR}{\relax\ifhmode\unskip\space\fi MR }
\providecommand{\MRhref}[2]{%
	\href{http://www.ams.org/mathscinet-getitem?mr=#1}{#2}
} \providecommand{\href}[2]{#2}

\end{document}